\documentclass[11pt]{amsart}
\usepackage{latexsym}
\usepackage{amssymb}
\usepackage{amsmath}

\newtheorem{theorem}{Theorem}[section]
\newtheorem{lemma}[theorem]{Lemma}
\newtheorem{proposition}[theorem]{Proposition}
\newtheorem{corollary}[theorem]{Corollary}
\newtheorem{definition}[theorem]{Definition}

\theoremstyle{definition}
\newtheorem{remark}[theorem]{Remark}
\newtheorem{example}[theorem]{Example}
\newtheorem{question}[theorem]{Question}

\newcommand\id{\mathop{\rm id}}

\newcommand\nph{\varphi}

\newcommand\K{\mathcal{K}}

\newcommand\oss{{\mathcal{S}}}
\newcommand\ost{{\mathcal{T}}}

\newcommand\omin{\otimes_{\rm min}}
\newcommand\omax{\otimes_{\rm max}}
\newcommand\oc{\otimes_{\rm c}}

\newcommand\comm{\mathop{\rm c}}

\newcommand\coisubset{\subseteq_{\rm coi}}   

\newcommand\cstar{{\rm C}^*}                              

\newcommand{\cl}[1]{\mathcal{#1}}
\newcommand{\bb}[1]{\mathbb{#1}}

\begin{document}

\title{Characterisations of the weak expectation property}

\author[D.~Farenick]{Douglas Farenick}
\address{Department of Mathematics and Statistics, University of Regina,
Regina, Saskatchewan S4S 0A2, Canada}
\email{douglas.farenick@uregina.ca}

\author[A.~S.~Kavruk]{Ali S.~Kavruk}
\address{Department of Mathematics, University of Illinois, Urbana, IL 61801, U.S.A.}
\email{kavruk@illinois.edu}

\author[V.~I.~Paulsen]{Vern I.~Paulsen}
\address{Department of Mathematics, University of Houston,
Houston, Texas 77204-3476, U.S.A.}
\email{vern@math.uh.edu}

\author[I.~G.~Todorov]{Ivan G.~Todorov}
\address{Pure Mathematics Research Centre, Queen's University Belfast, Belfast BT7 1NN, United Kingdom}
\email{i.todorov@qub.ac.uk}

\thanks{This work supported in part by NSERC (Canada), NSF (USA), and the Royal Society (United Kingdom)}

\date{25 June 2013}


\begin{abstract}
We use representations of operator systems as quotients 
to deduce various characterisations of the weak expectation property (WEP) for C$^*$-algebras.  By Kirchberg's work on WEP, 
these results give new formulations of Connes' embedding problem. 
\end{abstract}
\maketitle

\section{Introduction} 

In this paper
we 
deduce various characterisations of Lance's weak expectation property (WEP) for C*-algebras \cite{L}. 
Lance's original definition of WEP
requires that {\em every} faithful representation of the C*-algebra possesses a so-called  weak expectation or, equivalently, that the universal representation, 
which is somewhat cumbersome, possesses a weak expectation. Given a unital C*-algebra $\cl A$ and a faithful unital representation $\pi: \cl A \to B(\cl H),$ 
then a weak expectation is a completely positive map $\phi: B(\cl H) \to \pi(\cl A)^{\prime \prime}$ such that $\phi(\pi(a)) = \pi(a)$ for every $a \in \cl A.$

One advantage of our results is that they give new characterisations of WEP in terms of a fixed given representation. Thus, one is free to choose a 
preferred faithful representation of the C*-algebra to attempt to determine if it has WEP.
These results expand on earlier work of the first three authors \cite{farenick--kavruk--paulsen2011} and of the second author \cite{kavruk2012} that 
also obtained such representation-free characterisations of WEP.

One major motivation for the desire to obtain such a plethora of characterisations of WEP are the results of Kirchberg, who
proved that Connes' embedding conjecture is equivalent to determining if certain C*-algebras have WEP. Thus, a wealth of characterisations of 
WEP could help to resolve this conjecture.

Our technique is to first characterise WEP in terms of operator system tensor products with certain ``universal'' finite dimensional operator systems. 
These universal operator systems have no completely order isomorphic representations on finite dimensional spaces, but we then realise them 
as quotients of finite dimensional operator subsystems of matrix algebras. 
This leads to characterisations of the C*-algebras possessing  WEP 
as the C*-algebras for which 
these quotient maps remain quotient after tensoring with the algebra
(see Theorem \ref{p_eqmma} for a precise formulation). 
Thus, WEP is realised as an ``exactness'' property for these operator system quotients or, equivalently, 
as a ``lifting'' property from a quotient. 
Because the liftings lie in finite dimensional matrix algebras, the question 
of the existence or non-existence of liftings can be reduced to a question about the existence of liftings satisfying 
elementary linear constraints.

As in the work of the second author, many of these characterisations of WEP reduce to interpolation or decomposition 
properties of the C*-algebra of the type studied by F. Riesz in other contexts.
In 
ordered function space theory, or in general, in ordered topological lattice theory, 
the vast use of Riesz interpolation
and decomposition properties dates back to F. Riesz's studies in the late '30's \cite{Riesz}. 
The reader may refer to
\cite{CD} and the bibliography therein for broad applications of this concept. We also
refer to \cite{kavruk2012} for a non-commutative Riesz interpolation property that characterises WEP.  In this paper 
we give a characterisation of WEP in terms of a Riesz decomposition property.

\section{Operator System Preliminaries}\label{s_prel}

In this section, we introduce basic terminology and notation, and recall previous constructions and results
that will be needed in the sequel.
If $V$ is a vector space, we let $M_{n,m}(V)$ be the space of all $n$ by $m$ matrices with entries in $V$.
We set $M_n(V) = M_{n,n}(V)$ and $M_n = M_n(\mathbb{C})$.
We let $(E_{i,j})_{i,j}$ be the canonical matrix unit system in $M_n$.
For a map $\phi : V\to W$ between vector spaces, we let $\phi^{(n)} : M_n(V)\to M_n(W)$
be the $n$th ampliation of $\phi$ given by $\phi^{(n)}((x_{i,j})_{i,j}) = (\phi(x_{i,j}))_{i,j}$.
For a Hilbert space $\cl H$, we denote by $\cl B(\cl H)$ the
algebra of all bounded linear operators on $\cl H$.
An \emph{operator system} is a subspace $\cl S$ of
$\cl B(\cl H)$ for some Hilbert space $\cl H$ which contains the identity operator $I$ and is closed
under taking adjoints. The embedding of $M_n(\cl S)$
into $\cl B(\cl H^n)$ gives rise to the cone $M_n(\cl S)_+$ of all positive operators in $M_n(\cl S)$. The
family $(M_n(\cl S)_+)_{n\in \mathbb{N}}$ of cones is called the \emph{operator system structure} of $\cl S$.
Every complex $*$-vector space equipped with a family of matricial cones and an order unit satisfying natural axioms can, by virtue
of the Choi-Effros Theorem \cite{choi--effros1977}, be represented faithfully as an operator system acting on some Hilbert space.
When a particular embedding is not specified, the order unit of an operator system will be denoted by $1$.
A map $\phi : \cl S\to \cl T$ between operator systems is called \emph{completely positive} if $\phi^{(n)}$
positive, that is, $\phi^{(n)}(M_n(\cl S)_+)\subseteq M_n(\cl T)_+$, for every $n\in \mathbb{N}$.
A linear bijection $\phi : \cl S\to \cl T$ of operator systems $\oss$ and $\ost$ is a \emph{complete order isomorphism} if both $\phi$ and $\phi^{-1}$
are completely positive.
We refer the reader to
\cite{PaulsenBook} for further properties of operator systems and completely positive maps.

An \emph{operator system tensor product} $\oss\otimes_\tau\ost$ of operator systems $\oss$ and $\ost$
is an operator system structure on the algebraic tensor product $\oss\otimes\ost$ satisfying a set of natural axioms.
We refer the reader to \cite{kavruk--paulsen--todorov--tomforde2011}, where a detailed study of such tensor products
was undertaken.
Suppose that $\oss_1\subseteq\ost_1$ and $\oss_2\subseteq\ost_2$ are inclusions of operator systems. Let $\iota_j:\oss_j\rightarrow\ost_j$ denote the
inclusion maps $\iota_j(x_j)=x_j$ for $x_j\in\oss_j$, $j = 1,2$, so that the map
$\iota_1\otimes\iota_2:\oss_1\otimes\oss_2\rightarrow\ost_1\otimes\ost_2$ is a linear
inclusion of vector spaces. If $\tau$ and $\sigma$ are operator system structures on $\oss_1\otimes\oss_2$ and $\ost_1\otimes\ost_2$ respectively,
then we use the notation
\[
\oss_1\otimes_\tau\oss_2\,\subseteq_+\,\ost_1\otimes_\sigma\ost_2
\]
to express the fact that $\iota_1\otimes\iota_2:\oss_1\otimes_\tau\oss_2\rightarrow\ost_1\otimes_\sigma\ost_2$ is a
(unital) completely positive map. This notation is motivated by the fact that $\iota_1 \otimes \iota_2$ is a
completely positive map if and only if, for every $n,$ the cone $M_n(\oss_1\otimes_\tau\oss_2)_+$ is
contained in the cone $M_n(\ost_1\otimes_\sigma \ost_2)_+$.
If, in addition, $\iota_1\otimes\iota_2$ is a complete order isomorphism onto its range, then we write
\[
\oss_1\otimes_\tau\oss_2\coisubset\ost_1\otimes_\sigma\ost_2\,.
\]
In particular, if $\tau$ and $\sigma$ are two operator system structures on $\oss\otimes\ost$, then
\[
\oss\otimes_\tau\ost\,=\,\oss\otimes_\sigma\ost
\quad\mbox{means}\quad
\oss\otimes_\tau\ost\coisubset\oss\otimes_\sigma\ost \quad\mbox{and}\quad
\oss\otimes_\sigma\ost\coisubset\oss\otimes_\tau\ost \,.
\]

When $\oss_1\otimes_\tau\oss_2 \subseteq_+\oss_1\otimes_\sigma \oss_2,$ then we will also write $\tau \ge \sigma$
and say that $\tau$ majorises $\sigma$.

In the sequel, we will use extensively the following operator system tensor products
introduced in \cite{kavruk--paulsen--todorov--tomforde2011}:

\smallskip

(a) \emph{The minimal tensor product $\min$.} If $\cl S\subseteq \cl B(\cl H)$ and $\cl T\subseteq \cl B(\cl K)$, where
$\cl H$ and $\cl K$ are Hilbert spaces, then $\cl S\otimes_{\min}\cl T$ is the operator system arising from the
natural inclusion of $\cl S\otimes\cl T$ into $\cl B(\cl H\otimes\cl K)$.

\smallskip

(b) \emph{The maximal tensor product $\max$.} For each $n\in \mathbb{N}$, let
$D_n = \{A^*(P\otimes Q)A : A\in M_{n,km}(\mathbb{C}), P\in M_k(\cl S)_+, Q\in M_m(\cl T)_+\}$.
The Archimedanisation \cite{paulsen--todorov--tomforde2011} of the family $(D_n)_{n\in \mathbb{N}}$
of cones is an operator system structure on $\cl S\otimes\cl T$; the corresponding operator system
is denoted by $\cl S\otimes_{\max}\cl T$.

\smallskip

(c) \emph{The commuting tensor product ${\rm c}$.} By definition, an element $X\in M_n(\cl S\otimes\cl T)$
belongs to the positive cone $M_n(\cl S\oc\cl T)_+$ if
$(\phi\cdot\psi)^{(n)}(X)$ is a positive operator for all completely positive maps $\phi : \cl S\to \cl B(\cl H)$
and $\psi : \cl T\to \cl B(\cl H)$ with commuting ranges. Here, the linear map
$\phi\cdot\psi : \cl S\otimes \cl T\to \cl B(\cl H)$
is given by $\phi\cdot\psi(x\otimes y) = \phi(x)\psi(y)$, $x\in \cl S$, $y\in \cl T$.

\smallskip

The tensor products min, c, and max are functorial in the sense that
if $\tau$ denotes any of them, and $\phi : \cl S_1\to \cl S_2$ and $\psi : \cl T_1\to \cl T_2$ are
completely positive maps, then the tensor product map
$\phi\otimes\psi : \cl S_1\otimes_{\tau}\cl T_1\to \cl S_2\otimes_{\tau}\cl T_2$ is completely positive.
We will use repeatedly the following fact, extablished in \cite{kavruk--paulsen--todorov--tomforde2011}:
If $\cl S$ is an operator system and $\cl A$ is a C*-algebra, then 
$\cl S\otimes_{\rm c} \cl A = \cl S\otimes_{\max}\cl A$. 

The three tensor products
mentioned above satisfy the relations
\[
\oss\omax\ost\,\subseteq_+\,\oss\oc\ost\, \subseteq_+ \oss\omin\ost
\]
for all operator systems $\oss$ and $\ost$.

For every operator system $\cl S$, we denote by $\cl S^d$ the (normed space) dual of $\cl S$.
The space $M_n(\cl S^d)$ can be naturally identified with a subspace of the space
$\cl L(\cl S,M_n)$ of all linear maps from $\cl S$ into $M_n$. 
Taking the pre-image of the cone of all completely positive maps in $\cl L(\cl S,M_n)$, we obtain a family $(M_n(\cl S^d)_+)_{n\in \bb{N}}$
of matricial cones on $\cl S^d$. 
We have, in particular, that $(\cl S^d)_+$ consists of all positive functionals on $\cl S$; the elements $\phi\in (\cl S^d)_+$ with $\phi(1) = 1$
are called \emph{states} of $\cl S$. An important case arises when $\cl S$ is finite dimensional; in this case, $\cl S^d$ is an operator system
when equipped with the family of matricial cones just described and has an order unit 
given by any faithful state on $\cl S$ \cite[Corollary 4.5]{choi--effros1977}.

We now move to the notion of quotients in the operator system category.

\begin{definition} A linear subspace $\cl J\subseteq \oss$ of an operator system $\oss$ is called a \emph{kernel} if there is an operator system $\ost$
and a completely positive linear map $\phi:\oss\rightarrow\ost$ such that $\cl J=\ker\phi$.
\end{definition}

If $\cl J\subseteq\oss$ is kernel, then one may endow the $*$-vector space $\oss/\cl J$ with an operator system structure
such that the canonical quotient map $q_{\cl  J}:\oss\rightarrow\oss/\cl J$ is unital and completely positive
\cite{kavruk--paulsen--todorov--tomforde2010}.  An element $(x_{i,j} + \cl J)$ is positive in $M_n(\oss/\cl J)$ if and only if for every $\epsilon >0,$ there exist elements $y_{i,j} \in \cl J$ such that $(x_{i,j} + y_{i,j}) + \epsilon 1_n \in M_n(\oss)_+.$ 
Moreover, if $\cl J\subseteq\ker\phi$ for some completely positive map $\phi:\oss\rightarrow\ost$,
then there exists a completely positive map $\dot{\phi}:\oss/\cl J\rightarrow\ost$
such that $\phi=\dot{\phi}\circ q_\cl J$.
A \emph{null subspace} of $\cl S$ 
\cite{kavruk2011} is a subspace $\cl J$ which does not contain 
positive elements other than $0$. It was shown in \cite{kavruk2011} that 
every null subspace is a kernel.

\begin{definition} A unital completely positive map $\phi:\oss\rightarrow\ost$ is called a
\emph{complete quotient map} if the natural quotient map $\dot{\phi}:\oss/\ker\phi\rightarrow\ost$
is a complete order isomorphism.
\end{definition}

\begin{definition} Given an operator system $\ost$ an element $(t_{i,j}) \in M_n(\ost)$ will be called \emph{strongly positive} if there exists $\epsilon > 0$ such that $(t_{i,j}) - \epsilon 1_n \in M_n(\ost)_+.$ 
\end{definition}
Thus, an element of a C*-algebra is strongly positive if and only if it is positive and invertible and an element of an operator system is strongly positive if and only if its image under every unital completely positive map into a C*-algebra is positive and invertible.

We will write $(t_{i,j}) \gg 0$ to denote that $(t_{i,j})$ is strongly positive. Given two self-adjoint elements $(x_{i,j})$ and $(y_{i,j})$ we will write $(x_{i,j}) \gg (y_{i,j})$ or $(y_{i,j}) \ll (x_{i,j})$ to indicate that $(x_{i,j} - y_{i,j})$ is strongly positive.

The concept of strongly positive element leads to the following useful characterisation of complete quotient maps \cite[Proposition 3.2]{farenick--kavruk--paulsen2011}.
 
\begin{proposition}\label{cqchar} Let $\oss$ and $\ost$ be operator systems and let $\phi: \oss \to \ost$
be a unital completely positive surjection. Then $\phi$ is a complete quotient map if and only if for every $n$ every strongly positive element of $M_n(\ost)_+$ has a strongly positive pre-image.
\end{proposition} 

We will frequently use the following result \cite[Lemma 5.1]{farenick--paulsen2011}.

\begin{lemma}\label{l_fp}
Let $\cl R,\cl S, \cl T$ and $\cl U$ be operator systems and assume that
we are given linear maps
$\psi : \cl R \to \cl S$, $\theta : \cl S\to \cl T$,
$\mu : \cl R\to \cl U$ and $\nu : \cl U\to \cl T$, such that 
$\nu$ is a complete quotient map, $\mu$ is a complete order isomorphism, 
$\theta$ is a linear isomorphism, $\theta^{-1}$ is completely positive and 
$\theta\circ \psi = \nu\circ \mu$. 
Then $\psi$ is a complete quotient map if and only if $\theta$ is 
a complete order isomorphism. 
\end{lemma}

We recall the \emph{universal C*-algebra} $C^*_u(\cl S)$ 
of an operator system $\cl S$: it is the unique (up to a *-isomorphism)
C*-algebra containing $\cl S$ 
and having the property that whenever $\nph : \cl S\to \cl B(H)$ is a 
unital completely positive map, there exists a unique *-homomorphism 
$\pi : C^*_u(\cl S)\to \cl B(H)$ extending $\nph$.


\section{Characterisations of WEP via Group C$^*$-Algebras}

If $G$ is a discrete group, we let $C^*(G)$ denote, as is customary, the (full) group
C*-algebra of $G$. Of particular interest are free groups with 
finitely many, say $n$, or countably many, generators, which we denote
by
$\bb{F}_n$ and $\bb{F}_\infty$, respectively.

Kirchberg \cite[Proposition 1.1(iii)]{kirchberg1993} proved that  a C*-algebra $\cl A$  possesses the \emph{weak expectation property (WEP)} if and only if 
$C^*(\bb{F}_{\infty})\otimes_{\min} \cl A = C^*(\bb{F}_{\infty})\otimes_{\max} \cl A$. 
We will not use the Lance's original definition of WEPin this paper, only Kirchberg's characterisation. In this sense our paper is really about characterisations of C*-algebras 
that satisfy Kirchberg's tensor formula and it is only and it is only because of Kirchberg's theorem that these are characterisations of WEP.  
  
\emph{Kirchberg's Conjecture}, on the other hand, asserts that 
the C*-algebra $C^*(\bb{F}_{\infty})$ possesses WEP, i.e., 
that $C^*(\bb{F}_{\infty}) \otimes_{\min} C^*(\bb F_{\infty}) = C^*(\bb F_{\infty}) \otimes_{\max} C^*(\bb F_{\infty})$.
Kirchberg proved that Connes' Embedding Conjecture, which is a statement about type II$_1$-factors, is 
equivalent to the statement that $C^*(\bb{F}_{\infty})$ possesses WEP. Consequently, many author's refer to what we are calling Kirchberg's Conjecture as Connes' Embedding Problem.  
We prefer to distinguish between the two to stress that we are using Kirchberg's formulation.

A C$^*$-algebra $\cl A$ is said to have the \emph{quotient weak expectation property (QWEP)} if $\cl A$ is a quotient of a 
C$^*$-algebra $\cl B$ that has WEP. In many ways QWEP is a better behaved notion than WEP, as QWEP
enjoys a number of permanence properties that are not necessarily shared by WEP: see, for example, \cite[Proposition 4.1]{ozawa}.

Lifting properties will play an important role in the sequel. 
If $\cl J$ is an ideal in a unital C$^*$-algebra $\cl B$ 
and if $q_{\cl J}:\cl B\rightarrow\cl B /\cl J$ is the canonical quotient homomorphism, then 
a unital completely positive map $\phi:\cl S\rightarrow\cl B / \cl J$ of an operator system $\cl S$ into $\cl B/\cl J$ is said to be
\emph{liftable} if there is a unital completely positive map 
$\psi:\cl S\rightarrow \cl B$ such that $\phi=q_{\cl J}\circ\psi$.
A unital C$^*$-algebra $\cl A$ has the \emph{lifting property (LP)} if every 
unital completely positive map $\phi$ of $\cl A$ into $\cl B/\cl J$ is liftable, 
for every unital C*-algebra $\cl B$ and every closed ideal $\cl J\subseteq \cl B$.
A unital C$^*$-algebra $\cl A$ has the \emph{local lifting property (LLP)} if for every
unital completely positive map $\phi$ of $\cl A$ into $\cl B/\cl J,$ the restriction of $\phi$ to any finite dimensional operator subsystem 
$\cl S\subseteq \cl A$ is liftable.
In the 
operator system context, these lifting properties were studied in 
\cite{kavruk--paulsen--todorov--tomforde2010}.

If $\cl A_1$ and $\cl A_2$ are unital C*-algebras, 
we denote by $\cl A_1\ast\cl A_2$ the free product 
C*-algebra, amalgamated over the unit. The same notation 
is used for free products of groups.
The following result, which combines results of Boca \cite{Boca91} and Pisier \cite[Theorem 1.11]{pisier_intr}, will be useful for us 
in the sequel.

\begin{theorem}\label{th_boca}
Let $\cl A_1,\dots,\cl A_n$ be unital C*-algebras and $\varphi_i: \cl A_i \rightarrow \cl B(H)$ 
be unital completely positive maps, $i=1,\dots,n$.
Then there exists a unital completely positive 
map $\varphi : \cl A_1 * \cdots * \cl A_n\rightarrow \cl B(H)$ 
such that $\varphi|_{\cl A_i} = \varphi_i$. Furthermore, if each $\cl A_j$ is a separable C$^*$-algebra with LP, 
then $\cl A_1 * \cdots * \cl A_n$ has LP.
\end{theorem}

\begin{example}\label{lp groups} The following group C$^*$-algebras have property LP:
\begin{enumerate}
\item $C^*(\bb{F}_n)$, for all $n\in\mathbb N\cup\{\infty\}$;
\item $C^*(SL_2(\mathbb Z))$;
\item $C^*(\ast_{j=1}^n \bb{Z}_2 )$, where $\ast_{j=1}^n \bb{Z}_2$ is
the $n$-fold free product of $n$ copies of $\mathbb Z_2$, $n\in\mathbb N$.
\end{enumerate}
\end{example}

\begin{proof}
The fact that $\cstar(\bb{F}_n)$ and $C^*(SL_2(\mathbb Z))$ have the lifting property (LP) for all $n\in\mathbb N\cup\{\infty\}$ was established by Kirchberg \cite{kirchberg1993}.
There are alternate proofs for the assertion that $\cstar(\bb{F}_n)$ has LP: see \cite{ozawa,pisier_intr}, for example.

Suppose that 
$\phi: C^*(\bb{Z}_2)\to \cl B/\cl J$ is a unital completely 
positive map, where $\cl B$ is a unital C*-algebra and $\cl J\subseteq \cl B$ is a closed 
ideal. Let $b\in \cl B$ be a 
selfadjoint contractive lifting of $\phi(h)$, where $h$ is the 
generator of $C^*(\bb{Z}_2)$. Then the linear 
map $\tilde{\phi} : C^*(\bb{Z}_2)\to \cl B$
given by $\tilde{\phi}(h) = b$ is unital and completely positive 
(see, e.g. Proposition \ref{p_unip} (2)), which
is clearly a positive lifting of $\phi$.
To complete the proof observe that,
because $C^*(\ast_{j=1}^n \bb{Z}_2) = \ast_{j=1}^n C^*(\bb{Z}_2)$, 
Theorem \ref{th_boca} implies that  $C^*(\ast_{j=1}^n \bb{Z}_2 )$ has LP.
\end{proof}


We next record some observations that allow us to replace $\bb{F}_{\infty}$ in the formulation of 
Kirchberg's Conjecture by other discrete groups and, subsequently, to replace WEP by QWEP. 
Some of the results are certainly well-known, but we include their proofs for completeness.

\begin{proposition}\label{gwilldo} Let $G_1$ and $G_2$ be countable discrete groups that contain $\bb{F}_2.$ If $C^*(G_1) \otimes_{\min} C^*(G_2) = C^*(G_1) \otimes_{\max} C^*(G_2)$, then Kirchberg's Conjecture is true.
\end{proposition}
\begin{proof}
Since $\bb{F}_{2}$ contains $\bb{F}_{\infty}$ as a subgroup, it follows by our 
assumption that $G_1$ and $G_2$ do as well. 
By \cite[Proposition 8.8]{pisier_intr}, for $i=1,2$ there exists a canonical compete order embedding 
$\nph_i : C^*(\bb{F}_{\infty})\to C^*(G_i)$ and a unital completely positive projection $P_i : C^*(G_i)\to C^*(\bb{F}_{\infty})$.
Thus, there is a chain of completely positive maps
\begin{multline*}
C^*(\bb{F}_{\infty})\otimes_{\min} C^*(\bb F_{\infty})
\stackrel{\nph_1\otimes\nph_2}{\longrightarrow} C^*(G_1)\otimes_{\min} C^*(G_2) 
= \\C^*(G_1)\otimes_{\max}C^*(G_2) 
\stackrel{P_1\otimes P_2}{\longrightarrow} C^*(\bb{F}_{\infty})\otimes_{\max} C^*(\bb F_{\infty}),
\end{multline*}
and the result follows.
\end{proof}

\begin{definition} Consider the following two properties of a unital C$^*$-algebra $\cl A$:
\begin{enumerate}
\item $\cl A$ has WEP;
\item $C^*(G)\omin\cl A=C^*(G)\omax\cl A$.
\end{enumerate}
We say that a countable discrete group $G$ \emph{detects WEP} if (2) implies (1) and that $G$ \emph{characterises WEP} if (2) and (1) are equivalent.
\end{definition}

\begin{proposition}\label{p_dg}
Every countable discrete group $G$ that contains $\bb{F}_2$ as a subgroup detects WEP. If, in addition,
$C^*(G)$ has the local lifting property, then $G$ characterises WEP.
\end{proposition}
\begin{proof} Suppose that $G$ contains $\bb{F}_2$ and that 
$\cl A$ is a unital C$^*$-algebra for which $C^*(G)\omin\cl A=C^*(G)\omax\cl A.$
Since $\bb{F}_{2}$ contains $\bb{F}_{\infty}$ as a subgroup, it follows by our 
assumption that $G$ does so as well. 
By \cite[Proposition 8.8]{pisier_intr}, there exists a canonical compete order embedding 
$\nph : C^*(\bb{F}_{\infty})\to C^*(G)$ and a unital completely positive projection $P : C^*(G)\to C^*(\bb{F}_{\infty})$.
Thus, there is a chain of completely positive maps
$$C^*(\bb{F}_{\infty})\otimes_{\min} \cl A 
\stackrel{\nph\otimes\id}{\to} C^*(G)\otimes_{\min} \cl A 
= C^*(G)\otimes_{\max}\cl A 
\stackrel{P\otimes\id}{\to} C^*(\bb{F}_{\infty})\otimes_{\max} \cl A,$$
and so  $C^*(\bb{F}_{\infty})\otimes_{\min}\cl A = C^*(\bb{F}_{\infty})\otimes_{\max}\cl A$. 

If $C^*(G)$ has the local lifting property and $\cl A$ has WEP, then $C^*(G)\otimes_{\min}\cl A = C^*(G)\otimes_{\max}\cl A$, by
 \cite[Proposition 1.1(i)]{kirchberg1993}.
 \end{proof}
 
 \begin{example}\label{p_dg_eg} The following countable discrete groups characterise WEP:
 \begin{enumerate}
 \item $SL_2(\mathbb Z)$;
 \item $\ast_{j=1}^n \bb{Z}_2$, if $n \geq 3$.
 \end{enumerate}
 \end{example}
 
 \begin{proof}
Recall that Example \ref{lp groups} shows that $C^*( SL_2(\mathbb Z))$ and
$C^*(\ast_{j=1}^n \bb{Z}_2)$ have the lifting property, which is stronger than the local lifting property. 
Moreover, $SL_2(\mathbb Z)$ contains a copy of $\mathbb F_2$ (see, {\it e.g.}, \cite[p.486]{kirchberg1993}) 
as does 
$\ast_{j=1}^n \bb{Z}_2$ for $n\geq 3$ (see, {\it e.g.}, \cite{fritz}). 
Thus, Proposition \ref{p_dg} applies to each of these groups.
\end{proof}

\noindent {\bf Remark.} That the free product of $n$ copies of $\mathbb Z_2$ detects WEP 
was also established by T.~Fritz in \cite{fritz} 
using a different method.  

Since $SL_3(\bb Z)$ also contains $\bb F_2$ this group detects WEP, but it is not known if $C^*(SL_3(\bb Z))$ has the lifting property. It would be interesting to know whether $SL_3(\mathbb Z)$ characterises WEP.
More generally, since
every countable discrete group $G$ that contains $\mathbb F_2$ as a subgroup
and has the local lifting property characterises WEP (Proposition \ref{p_dg}) it would be interesting to know if
these two sufficient conditions are in fact necessary. A small step in this direction is the following proposition.

\begin{proposition}\label{tits alt} If a finitely generated discrete subgroup $G$ of $GL_n(\mathbb C)$ detects WEP, then $G\supseteq\mathbb F_2$ and $C^*(G)$ is a non-exact C$^*$-algebra.
\end{proposition}

\begin{proof} If a finitely generated discrete subgroup $G$ of $GL_n(\mathbb C)$ does not contain $\mathbb F_2$ as a subgroup, 
then $G$ contains a normal subgroup $H$ such that $H$ is solvable and $G/H$ is finite (this is the ``Tits Alternative'').
As solvable and finite groups are amenable, $G$ is an extension of an amenable group by an amenable group and is, therefore, amenable. Hence, $C^*(G)$ is nuclear and therefore $G$ cannot detect WEP.
The only alternative, thus, is that $G$ contains $\mathbb F_2$ as a subgroup.

Because $G$ is a finitely generated linear group, 
every finitely generated subgroup of $G$ is maximally almost periodic (see, {\it e.g.}, \cite{kirchberg1993}). 
Thus, if $C^*(G)$ were exact, then it would in fact be nuclear
\cite[Theorem 7.5]{kirchberg1993}; but then $G$ cannot detect WEP.
\end{proof}

Observe that the argument of Proposition \ref{tits alt} applies to any countable discrete group $G$ 
for which the Tits Alternative holds and it yields the inclusion $G\supseteq\mathbb F_2$.

\begin{proposition}\label{keq} The following statements are equivalent for a countable discrete group $G$
that contains $\bb F_2$ and such that C$^*(G)$ has the local lifting property: 
\begin{enumerate}
\item\label{un} Kirchberg's Conjecture is true;
\item\label{deux} $C^*(G)$ has WEP;
\item\label{trois} $C^*(G)$ has QWEP.
\end{enumerate}
\end{proposition}
\begin{proof} 
\eqref{un} $\Rightarrow $\eqref{deux}.
If Kirchberg's conjecture is true, then $C^*(\bb F_{\infty})$ has WEP and, hence by Proposition~\ref{p_dg}, $C^*(G) \otimes_{\min} C^*(\bb F_{\infty}) = C^*(G) \otimes_{\max} C^*(\bb F_{\infty})$ and so by Kirchberg's theorem, $C^*(G)$ has WEP. 

\eqref{deux} $\Rightarrow $\eqref{trois}. Trivial

\eqref{trois} $\Rightarrow $\eqref{un}.
Assume $C^*(G)$ has QWEP.
Let $C^*(G)\subseteq\ C^*(G)^{**}\subseteq\cl B(\cl H_u)$, where $\cl H_u$ is the 
Hilbert space of the universal representation of $C^*(G)$.
Then $C^*(G)^{**}$ has QWEP also \cite[Corollary 3.3(v)]{kirchberg1993}. Since $C^*(G)$ has the lifting property, 
$C^*(G)$ is a unital C$^*$-subalgebra with LLP 
of a von Neumann subalgebra $C^*(G)^{**}$ of $\cl B(\cl H_u)$ with QWEP. Therefore, by 
\cite[Corollary 3.8(ii)]{kirchberg1993}, there is a unital completely positive map 
$\phi:\cl B(\cl H_u)\rightarrow C^*(G)^{**}$ such that $\phi(a)=a$ for every $a\in C^*(G)$. Thus, 
$C^*(G)$ has WEP. Now by Kirchberg's criterion for WEP \cite[Proposition 1.1(iii)]{kirchberg1993}, $C^*(G)\omin C^*(\bb F_\infty)=C^*(G)\omax C^*(\bb F_\infty)$.
Therefore, Propositon~\ref{gwilldo} implies that Kirchberg's Conjecture is true.
\end{proof}

\begin{corollary} Kirchberg's conjecture is true if and only if $C^*(\bb{F}_2)$ is a quotient of a C$^*$-algebra with WEP.
\end{corollary}

\section{Characterisations of WEP via Noncommutative $n$-Cubes}\label{s_pl}

Let $\bb{F}_n$ be the free group on $n$ generators
and $\ast_{j=1}^n \bb{Z}_2$, be the $n$-fold free product of 
the group $\bb{Z}_2$ of two elements ($n\in \bb{N}$). 
Following \cite{kavruk--paulsen--todorov--tomforde2010}, 
we let
$$\cl S_n = {\rm span}\{1,u_i,u_i^* : 1 = 1,\dots,n\}\subseteq C^*(\bb{F}_n),$$
where $u_1,\dots,u_n$ are the generators of $\bb{F}_n$ viewed as 
elements of $C^*(\bb{F}_n)$ and $u^{-1}_{i} = u_i^*$, 
$i = 1,\dots,n$. 
We also let \cite{fkpt} $NC(n)$ be the operator system
$$NC(n) = {\rm span}\{1,h_i : i = 1,\dots,n\}\subseteq C^*(\ast_{j=1}^n \bb{Z}_2),$$
where $h_1,\dots,h_n$ are the canonical generators of $\ast_{j=1}^n \bb{Z}_2$ 
(that is, $h_i$ is the non-trivial element of the $i$th copy of $\bb{Z}_2$
in $\ast_{j=1}^n \bb{Z}_2$, viewed as an element of $C^*(\ast_{j=1}^n \bb{Z}_2)$). 
The operator system $NC(n)$ is called the \emph{operator system of the noncommutative $n$-cube}. 

We note that $\cl S_n$ and $NC(n)$ are characterised by the following universal properties 
\cite{fkpt}, \cite{kavruk--paulsen--todorov--tomforde2010}:

\begin{proposition}\label{p_unip}
Let $\cl S$ be operator system. 
\begin{enumerate}
\item If $x_1,\dots,x_n\in \cl S$ are contractions then there exists a (unique) unital completely positive map 
$\nph : \cl S_n\to \cl S$ such that $\nph(u_i) = x_i$, $i = 1,\dots,n$.
\item If $x_1,\dots,x_n\in \cl S$ are selfadjoint contractions then there exists a (unique) unital completely positive map 
$\nph' : NC(n)\to \cl S$ such that $\nph'(h_i) = x_i$, $i = 1,\dots,n$.
\end{enumerate}
\end{proposition}

In fact, in \cite{fkpt}, $NC(n)$ is originally defined via this above universal property.

By \cite[Proposition 5.7]{fkpt}, the linear map $\psi : \cl S_n\to NC(n)$ given by $\psi(1) = 1$ and $\psi(u_i) = \psi(u_i^*) = h_i$,
$i = 1,\dots,n$, is a complete quotient map. 
Note that 
$$\ker\psi = \cl S_n^0 \stackrel{def}{=} \left\{\sum_{i=-n}^n \lambda_i u_i : \lambda_0 = 0, \lambda_i + \lambda_{-i} = 0, i = 1,\dots,n\right\}.$$

Let $\cl T_{n+1}$ be the tridiagonal operator system in $M_{n+1}$, that is, 
$\cl T_{n+1} = {\rm span}\{E_{i,j} : |i - j|\leq 1\}$, where 
$\{E_{i,j} \}_{i,j}$ denote the standard matrix unit system. 
By \cite[Theorem 4.2]{farenick--paulsen2011}, the linear map 
$\phi : \cl T_{n+1}\to \cl S_n$ given by $\phi(E_{i,j}) = \frac{1}{n+1} u_{j-i}$, is a complete quotient map.
Note that 
$$\ker\phi =  \left\{\sum_{i=1}^{n+1} \lambda_i E_{i,i} : \sum_{i=1}^{n+1} \lambda_i = 0\right\}.$$

The symbols $\phi$ and $\psi$ will be used 
to denote the maps introduced above. 

Let 
$$\K_{n+1} \stackrel{def}{=} \left\{\sum_{i=1}^{n+1} a_i E_{i,i} + \sum_{i=1}^n (b_i E_{i,i+1} - b_i E_{i+1,i}) : \sum_{i=1}^{n+1} a_i = 0\right\} \subseteq \cl T_{n+1}.$$

\begin{proposition}\label{p_quo}
The map $\rho \stackrel{def}{=} \psi\circ\phi : \cl T_{n+1}\to NC(n)$ is a complete quotient map with kernel $\K_{n+1}$. 
\end{proposition}
\begin{proof}
Since $\phi$ and $\psi$ are complete quotient maps, their composition is also a complete quotient 
(indeed, if $X\in M_k(NC(n))$ is strongly positive then by \cite[Proposition 3.2]{farenick--kavruk--paulsen2011}
it has a strongly positive lifting $Y\in M_k(\cl S_n)$ and, by the same result, $Y$ has a strongly positive lifting in $M_k(\cl T_{n+1})$). 
The image of an element $u = \sum_{i=1}^{n+1} a_i E_{i,i} + \sum_{i=1}^n (b_i E_{i,i+1} + c_i E_{i+1,i})$ under 
$\psi\circ\phi$ is 
$$\psi(\phi(u)) = \frac{1}{n+1} \sum_{i=1}^{n+1} a_i 1 + \frac{1}{n+1}\sum_{i=1}^n (b_i + c_i)h_i;$$
thus, $\psi(\phi(u)) = 0$ precisely when $\sum_{i=1}^{n+1} a_i  = 0$ and $b_i + c_i = 0$, $i = 1,\dots,n$.
\end{proof}

\begin{theorem}\label{p_eqmma}
The following statements are equivalent for a C*-algebra $\cl A$:
\begin{enumerate}
\item\label{one} \ $NC(n)\otimes_{\min}\cl A = NC(n)\otimes_{\max} \cl A$;
\item\label{two} the map $\rho\otimes_{\min}\id : \cl T_{n+1}\otimes_{\min} \cl A\to NC(n)\otimes_{\min} \cl A$ is a complete quotient map.
\end{enumerate}
Moreover, if $n \geq 3$, then these statements are also equivalent to:
\begin{enumerate}
\item[{(3)}]\label{three} $\cl A$ possesses WEP. 
\end{enumerate}
\end{theorem}
\begin{proof}
The map 
$\rho\otimes_{\min}\id$ 
is completely positive by the functoriality of $\min$. 
On the other hand, 
$\rho\otimes_{\max}\id : \cl T_{n+1}\otimes_{\max} \cl A\to NC(n)\otimes_{\max} \cl A$ 
is a complete quotient map by \cite[Proposition 1.6]{farenick--paulsen2011}. 
By \cite[Proposition 4.1]{farenick--paulsen2011}, 
and the fact that $\cl T_{n+1}\otimes_{\rm c}\cl A = \cl T_{n+1}\otimes_{\max}\cl A$ 
(see \cite[Proposition 6.7]{kavruk--paulsen--todorov--tomforde2011}),
the canonical map 
$\cl T_{n+1}\otimes_{\max}\cl A \to\cl T_{n+1}\otimes_{\min}\cl A$ is 
a complete order isomorphism. 
It follows from Lemma \ref{l_fp} that \eqref{one} and \eqref{two} are equivalent.

Suppose $n\geq  3$ and set $\cl B = C^*(\ast_{j=1}^n\bb{Z}_2)$. 
Assume (3) holds.
The group $\ast_{j=1}^n\bb{Z}_2$ contains $\bb{F}_2$ (see, for example, \cite{fritz}) and 
hence, by Proposition \ref{p_dg} and Example \ref{lp groups} (3),
$\cl B\otimes_{\min} \cl A = \cl B\otimes_{\max}\cl A$. 
By the injectivity of $\min$, we have $NC(n)\otimes_{\min} \cl A\coisubset \cl B\otimes_{\min} \cl A$, 
and by \cite[Lemma 6.2]{fkpt}, 
$NC(n)\otimes_{\max} \cl A\coisubset \cl B\otimes_{\max} \cl A$. 
It now follows that $NC(n)\otimes_{\min}\cl A = NC(n)\otimes_{\max}\cl A$. 

Finally, assume \eqref{one}.
By \cite[Proposition 2.2]{fkpt},
$\cl B = C_e^*(NC(n))$. 
The natural inclusion of vector spaces 
$NC(n)\otimes_{\min}\cl A\to \cl B\otimes_{\max}\cl A$ is completely positive as it is 
the composition of the completely positive maps
$NC(n)\otimes_{\min}\cl A\to NC(n)\otimes_{\max}\cl A$ and 
$NC(n)\otimes_{\max}\cl A\to \cl B\otimes_{\max}\cl A$. 
It follows from 
\cite[Proposition 9.5]{kavruk--paulsen--todorov--tomforde2010} that 
the natural map $\cl B\otimes_{\min}\cl A\to \cl B\otimes_{\max}\cl A$ is completely 
positive and hence 
$\cl B\otimes_{\min}\cl A = \cl B\otimes_{\max}\cl A$.
Proposition \ref{p_dg} now shows that $\cl A$ has WEP. 
\end{proof}

\begin{corollary}\label{p_llp}
The operator system $NC(n)$ has the 
lifting property for every $n\in \bb{N}$ 
and $NC(n)\otimes_{\min}\cl B(H) = NC(n)\otimes_{\max}\cl B(H)$ for every $n\in \bb{N}$ and
every Hilbert space $H$.
\end{corollary}
\begin{proof} If $\varphi:NC(n) \to \cl B/\cl J$ is a unital completely positive map, then the images of the generators of $NC(n)$ are hermitian contractions in $\cl B/\cl J.$ 
But each hermitian contraction in $\cl B/\cl J$ can be lifted to a hermitian contraction in $\cl B$ and these elements induce a unital completely positive lifting of $\phi$ by 
Proposition~\ref{p_unip}(2). 
The tensor equality follows from Theorem \ref{p_eqmma} and the fact that $\cl B(H)$ possesses WEP.
\end{proof}

\begin{corollary}\label{c_nc2}
The map $\rho\otimes_{\min}\id : \cl T_3\otimes_{\min}\cl A\to NC(2)\otimes_{\min}\cl A$ is a complete quotient map for every unital C$^*$-algebra $\cl A$. 
Hence, if $A_0,A_1,A_2\in M_k(\cl A)$ are such that $1\otimes A_0 + h_1\otimes A_1 + h_2\otimes A_2$ is strongly positive in $NC(2)\otimes_{\min}M_k(\cl A)$, then
there exist elements $A,B,C,X,Y\in M_k(\cl A)$ with 
$A+B+C = A_0$, $X + X^* = A_1$, $Y + Y^* = A_2$ such that the matrix 
$$\left(\begin{matrix} A & X & 0\\ X^* &  B & Y\\ 0 & Y^* & C\end{matrix}\right)$$ is strongly positive in $M_{3k}(\cl A)$. 
\end{corollary}
\begin{proof}
By \cite[Theorem 6.3]{fkpt}, $NC(2)$ is $(\min,\comm)$-nuclear and now
\cite[Proposition 6.7]{kavruk--paulsen--todorov--tomforde2011}
shows that $NC(2)\otimes_{\min}\cl A = NC(2)\otimes_{\max}\cl A$. 
By Theorem \ref{p_eqmma}, 
$\rho\otimes_{\min}\id : \cl T_{3}\otimes_{\min} \cl A\to NC(2)\otimes_{\min} \cl A$ is a complete quotient map. 
Hence, if $u = 1\otimes A_0 + h_1\otimes A_1 + h_2\otimes A_2$ is strongly positive in $NC(2)\otimes_{\min}M_k(\cl A)$, then
\cite[Proposition 3.2]{farenick--kavruk--paulsen2011} and Proposition~\ref{cqchar} implies that
there exist $A,B,C,X,Y\in M_k(\cl A)$ such that 
$$v = E_{1,1}\otimes A + E_{2,2}\otimes B + E_{3,3}\otimes C + E_{1,2}\otimes X + E_{2,1}\otimes X^* 
+ E_{2,3}\otimes Y + E_{3,2}\otimes Y^*$$ is strongly positive in $\cl T_3\otimes_{\min}\cl A$ and 
$$u = (\rho\otimes\id)^{(k)} (v) = \frac{1}{3} (1\otimes (A+B+C) + h_1\otimes(X + X^*) + h_2\otimes (Y + Y^*)).$$
It follows that $\frac{1}{3}(A+B+C) = A_0$, $\frac{1}{3}(X + X^*) = A_1$, $\frac{1}{3}(Y + Y^*) = A_2$. 
Rescaling $A,B,C,X$ and $Y$ by a factor of $\frac{1}{3}$ shows the claim. 
\end{proof}

\begin{corollary}\label{c_nc3}
The following statements are equivalent for a unital C$^*$-algebra $\cl A$:
\begin{enumerate}
\item\label{c_nc3-1} $\cl A$ has WEP;
\item\label{c_nc3-2} 
whenever $A_0,A_1,A_2,A_3\in M_k(\cl A)$ are such that 
$1\otimes A_0 + h_1\otimes A_1 + h_2\otimes A_2 + h_3\otimes A_3$ is strongly positive in $NC(3)\otimes_{\min}M_k(\cl A)$, then
there exist elements $A,B,C,D,X,Y,Z\in M_k(\cl A)$ with 
$A+B+C + D = A_0$, $X + X^* = A_1$, $Y + Y^* = A_2$ and $Z + Z^* = A_3$ 
such that the matrix 
$$\left(\begin{matrix} 
A & X & 0 & 0\\ 
X^* &  B & Y & 0\\ 
0 & Y^* & C & Z\\
0 & 0 & Z^* & D
\end{matrix}\right)$$ is strongly positive in $M_{4k}(\cl A)$. 
\end{enumerate}
\end{corollary}
\begin{proof}
As in the proof of Corollary \ref{c_nc2}, one can see that \eqref{c_nc3-2}  is equivalent to the 
canonical map 
$\rho\otimes_{\min}\id : \cl T_4\otimes_{\min} \cl A \to NC(3)\otimes_{\min}\cl A$ 
being a complete quotient map.
By Theorem \ref{p_eqmma}, the latter condition is equivalent to 
$\cl A$ having WEP.
\end{proof}

We next include a characterisation of WEP in terms of liftings of strongly positive elements. 
We recall that the numerical radius $w(x)$ of an element $x$ of an operator system $\cl S$
is given by $w(x) = \sup\{|f(x)| : f \mbox{ a state of }\cl S\}$.

\begin{lemma}\label{l_dsu}
If $\cl S,\cl T$ and $\cl R$ are operator systems and $\tau\in \{\min,\comm\}$, then
$((\cl S\oplus\cl T)\otimes_{\tau}\cl R)_+ = ((\cl S\otimes_{\tau}\cl R)\oplus (\cl T\otimes_{\tau}\cl R))_+$.
\end{lemma}
\begin{proof}
It is clear that there is a linear identification 
$\iota : (\cl S\oplus\cl T)\otimes\cl R \to (\cl S\otimes\cl R)\oplus (\cl T\otimes\cl R)$.
Suppose that $u\in ((\cl S\otimes_{\comm}\cl R)\oplus (\cl T\otimes_{\comm}\cl R))_+$, 
and write $u = (u_1,u_2)$, with 
$u_1\in (\cl S\otimes_{\comm}\cl R)_+$ and $u_2\in (\cl T\otimes_{\comm}\cl R)_+$. 
Let $f : \cl S\oplus \cl T\to \cl B(H)$ and $g : \cl R\to \cl B(H)$ 
be completely positive maps with commuting ranges. 
Let $f_1 = f|_{\cl S}$ and $f_2 = f|_{\cl T}$. Then $f_1$ and $f_2$ are 
completely positive and hence 
$f_1\cdot g(u_1) \geq 0$ and $f_2\cdot g(u_2) \geq 0$. 
But then $f\cdot g(\iota^{-1}(u)) = f_1\cdot g(u_1) + f_2\cdot g(u_2) \geq 0$. 

Conversely, assume that $u\in ((\cl S\oplus\cl T)\otimes_{\comm}\cl R)_+$. Write $\iota(u) = (u_1,u_2)$. 
If $f_1 : \cl S\to \cl B(H)$ and $g : \cl R\to \cl B(H)$ are completely positive maps with commuting ranges, 
then the map $f : \cl S\oplus \cl T\to \cl B(H)$ given by $f((x,y)) = f_1(x)$ is completely positive and 
hence $f_1\cdot g(u_1) = f\cdot g(u) \geq 0$. Thus, $u_1\in (\cl S\otimes_{\comm}\cl R)_+$; 
similarly, $u_2\in (\cl T\otimes_{\comm}\cl R)_+$. 

The statement regarding $\min$ is immediate from the 
injectivity of this tensor product.
\end{proof}

\begin{theorem}\label{th_st}
The following statements are equivalent for a unital C$^*$-algebra $\cl A$:
\begin{enumerate}
\item\label{th_st-1} $\cl A$ has WEP;
\item\label{th_st-2} whenever 
$$X = 1\otimes A_0 + u_1\otimes A_1 + u_1^*\otimes A_1^* + u_2\otimes A_2 + u_2^*\otimes A_2^*$$ is a strongly positive 
element of $M_k(\cl S_2\otimes_{\min}\cl A)$, where $A_0,A_1,A_2\in M_k(\cl A)$, there exist strongly positive elements $B,C\in M_k(\cl A)$
such that 
$$A_0 = \frac{1}{2}(B + C), \ \ w(B^{-\frac{1}{2}}A_1B^{-\frac{1}{2}}) < \frac{1}{2} \ \ \mbox{ and } \ \ 
w(C^{-\frac{1}{2}}A_2C^{-\frac{1}{2}}) < \frac{1}{2}.$$
\item\label{th_st-3} whenever $A_1, A_2 \in M_k(\cl A)$ satisfy $w(A_1,A_2) < 1/2$ then there exist positive invertible elements $B,C \in M_k(\cl A)$ such that 
$$\frac{1}{2}(B+C)=I, \, w(B^{-\frac{1}{2}}A_1B^{-\frac{1}{2}}) < 1/2 \mbox{ and } w(C^{-\frac{1}{2}}A_2C^{-\frac{1}{2}}) < 1/2.$$
\end{enumerate}
\end{theorem}
\begin{proof} We first prove the equivalence of (1) and (2).
Let $\cl J = {\rm span}\{(1,-1)\}\subseteq \cl S_1\oplus \cl S_1$.
It follows from \cite[Corollary 4.4]{kavruk2012} 
and \cite[Proposition 4.7]{kavruk2012} that 
$\cl J$ is a kernel and $\cl S_2 = (\cl S_1\oplus \cl S_1)/\cl J$. 
Let $q : \cl S_1\oplus \cl S_1\to \cl S_2$ be the corresponding (complete) quotient map.
By \cite[Proposition 3.3]{fkpt}, $\cl S_1$ is $(\min,\comm)$-nuclear and hence 
$\cl S_1\otimes_{\min}\cl A = \cl S_1\otimes_{\max}\cl A$. 
On the other hand, for every operator system $\cl S$ and every unital C*-algebra $\cl B$, 
we have that $M_k(\cl S\otimes_{\min}\cl B) = \cl S\otimes_{\min} M_k(\cl B)$ and 
$M_k(\cl S\otimes_{\max}\cl B) = \cl S\otimes_{\max} M_k(\cl B)$, $k\in \bb{N}$. 
It now follows from Lemma \ref{l_dsu} that 
$(\cl S_1\oplus\cl S_1)\otimes_{\min}\cl A = (\cl S_1\oplus\cl S_1)\otimes_{\max}\cl A$.
In the diagram 
$$\begin{matrix} 
(\cl S_1\oplus \cl S_1)\otimes_{\min}\cl A & = & (\cl S_1\oplus \cl S_1)\otimes_{\max}\cl A\\
\downarrow & & \downarrow\\
\cl S_2\otimes_{\min}\cl A & \leftarrow & \cl S_2\otimes_{\max}\cl A,
\end{matrix}$$
the right arrow denotes a complete quotient map by \cite[Proposition 1.6]{farenick--paulsen2011}, 
while the left arrow denotes the completely positive map $q\otimes_{\min}\id$
arising from the functoriality of $\min$. 
By Lemma \ref{l_fp}, 
$\cl S_2\otimes_{\min}\cl A = \cl S_2\otimes_{\max}\cl A$ if and only if 
$q\otimes_{\min}\id$ is a complete quotient map. 
By \cite[Theorem 5.9]{kavruk2011}, 
it suffices to show that \eqref{th_st-2} is equivalent to $q\otimes_{\min}\id$ being a complete quotient map.

To this end, suppose that $q\otimes_{\min}\id$ is a complete quotient and 
let $X$ be the strongly positive element of $M_k(\cl S_2\otimes_{\min}\cl A)$ given in \eqref{th_st-2}. 
By \cite[Proposition 3.2]{farenick--kavruk--paulsen2011}, 
there exists a strongly positive element 
$Y\in M_k((\cl S_1\oplus \cl S_1)\otimes_{\min}\cl A )$ with $(q\otimes_{\min}\id)^{(k)}(Y) = X$. 
By virtue of Lemma \ref{l_dsu}, write $Y = (Y_1,Y_2)$, where $Y_1$ and $Y_2$ are strongly positive elements of 
$\cl S_1\otimes_{\min}M_k(\cl A)$. 
Write $Y_1 = 1\otimes B + \zeta\otimes B_1 + \bar{\zeta}\otimes B_2$ and 
$Y_2 = 1\otimes C + \zeta\otimes C_1 + \bar{\zeta}\otimes C_2$, where we have denoted by $\zeta$ the 
generator of $\cl S_1$, viewed as the identity function on the unit circle $\bb{T}$. 
It follows that
$$1\otimes \frac{1}{2}(B + C) + u_1 \otimes B_1 + u_1^*\otimes B_2 + u_2\otimes C_1 + u_2^*\otimes C_2 = X,$$
which shows that $\frac{1}{2}(B + C) = A_0$, $B_1 = B_2^* = A_1$, $C_1 = C_2^* = A_2$.

Suppose that 
$Y_1\geq \delta 1$. Then, for every $z\in \bb{T}$ we have that 
$B + z A_1 + \bar{z} A_1^* \geq \delta I$ in $M_k(\cl A)$. Taking $z = \pm 1$, we see that 
$B\geq \delta 1$ and hence $B$ is invertible. Thus, 
$I + z B^{-\frac{1}{2}}A_1B^{-\frac{1}{2}} + \bar{z} B^{-\frac{1}{2}}A_1^*B^{-\frac{1}{2}} 
\geq \frac{\delta}{\|B\|} I$ in $M_k(\cl A)$,
for every $z\in \bb{T}$.
By \cite[Theorem 1.1]{farenick--kavruk--paulsen2011}, this implies that
$w(B^{-\frac{1}{2}}A_1B^{-\frac{1}{2}}) < \frac{1}{2}$. Similarly, $C$ is invertible and 
$w(C^{-\frac{1}{2}}A_2C^{-\frac{1}{2}}) < \frac{1}{2}$.

Conversely, if \eqref{th_st-2} is satisfied then reversing the steps in the previous two paragraphs shows that 
the element $Y = (Y_1,Y_2)$ is a strongly positive lifting of $X$. 
By \cite[Proposition 3.2]{farenick--kavruk--paulsen2011} and Proposition~\ref{cqchar}, 
$q\otimes_{\min}\id$ is a complete quotient map.   This proves the equivalence of (1) and (2).

We now show that (2) and (3) are equivalent.

Recall that $w(A_1,A_2) < 1/2$ if and only if 
$$I \otimes I + A_1 \otimes u_1 + A_1^* \otimes u_1^* + A_2 \otimes u_2 + A_2^* \otimes u_2^*$$
is strictly positive. From this we see that  (2) implies (3).  Conversely, if (3) holds then (2) holds for the case that $A_0=I.$ For the general case, use that fact that the strict positivity implies that $A_0$ is positive and invertible and conjugate by $A_0^{-\frac{1}{2}}.$
\end{proof}

\begin{theorem}\label{th_eqnc}
The following statements are equivalent for $\cl A = C^*(\ast_{j=1}^m \bb{Z}_2)$:
\begin{enumerate}
\item\label{th_eqnc-1}
$\rho\otimes_{\min}\id : \cl T_{n+1}\otimes_{\min} \cl A\to NC(n)\otimes_{\min} \cl A$ is a complete quotient map;
\item\label{th_eqnc-2}
$NC(n)\otimes_{\min} NC(m) = NC(n)\otimes_c NC(m)$.
\end{enumerate}
Moreover, if $n,m \geq 3$, then these statements are equivalent to:
\begin{enumerate}
\item[{(3)}]\label{th_eqnc-3}
Kirchberg's Conjecture holds true. 
\end{enumerate}
\end{theorem}
\begin{proof}
\eqref{th_eqnc-1} $ \Rightarrow$ \eqref{th_eqnc-2}.
By Theorem \ref{p_eqmma}, $NC(n)\otimes_{\min}\cl A = NC(n)\otimes_{\max} \cl A$. 
On the other hand, $NC(n)\otimes_{\min} NC(m)\coisubset NC(n)\otimes_{\min}\cl A$ by the 
injectivity of $\min$, while 
$NC(n)\otimes_c NC(m)\coisubset NC(n)\otimes_{\max} \cl A$ by \cite[Lemma 6.2]{fkpt}.
It follows that $NC(n)\otimes_{\min} NC(m) = NC(n)\otimes_c NC(m)$.

\eqref{th_eqnc-2}  $\Rightarrow$ \eqref{th_eqnc-1}.
Set $\cl B = C^*(\ast_{j=1}^n \bb{Z}_2)$. By \cite[Lemma 6.2]{fkpt}, we have that
$NC(n)\otimes_c NC(m)\coisubset \cl B\otimes_{\max}\cl A$. The assumption implies that 
the linear embedding $\beta : NC(n)\otimes_{\min} NC(m)\to \cl B\otimes_{\max}\cl A$ 
is completely positive. On the other hand, $NC(n)\otimes_{\min} NC(m)\coisubset \cl B\otimes_{\min}\cl A$ 
and $\beta(u\otimes v)$ is unitary, for all canonical unitary generators $u$ (resp. $v$) 
of the operator system $NC(n)$ (resp. $NC(m)$). 
Since $u\otimes v$ for such $u$ and $v$ generate $\cl B\otimes_{\min}\cl A$, 
\cite[Lemma 9.3]{kavruk--paulsen--todorov--tomforde2010} implies that $\phi$ has an extension 
to a *-homomorphism $\pi : \cl B\otimes_{\min}\cl A\to \cl B\otimes_{\max}\cl A$. Thus, 
every positive element of $M_k(NC(n)\otimes_{\min}\cl A)$ is sent via $\pi^{(k)}$ to a positive element of
$M_k(NC(n)\otimes_{\max}\cl A)$ and, by Theorem \ref{p_eqmma}, $\rho\otimes_{\min}\id$ is 
a complete quotient map. 

Suppose that $n,m\geq 3$. Assuming \eqref{th_eqnc-2}, we have seen that $\cl B\otimes_{\min}\cl A = \cl B\otimes_{\max}\cl A$. 
By \cite[Corollary C.4]{fritz}, $C^*(\bb{F}_2)\otimes_{\min} C^*(\bb{F}_2) = C^*(\bb{F}_2)\otimes_{\max} C^*(\bb{F}_2)$, and hence 
Kirchberg's Conjecture holds.  
Conversely, if Kirchberg's Conjecture holds then 
$\cl S_n\otimes_{\min}\cl S_m = \cl S_n\otimes_{\max}\cl S_m$.
Denoting for a moment by $\psi_n$ (resp. $\psi_m$) the 
canonical quotient map from $\cl S_n$ onto $NC(n)$ introduced 
after Proposition \ref{p_unip}, 
we have, by \cite[Proposition 1.6]{farenick--paulsen2011}, that 
$\psi_n\otimes\psi_m : \cl S_n\otimes_{\max}\cl S_m
\to NC(n)\otimes_{\max} NC(m)$ is a complete quotient map. 
Let $\gamma_n : NC(n)\to \cl S_n$ be the linear map given by 
$\gamma_n(h_i) = \frac{h_i + h_i^*}{2}$, $i = 1,\dots,n$. 
By \cite[Proposition 5.7]{fkpt}, $\gamma_n$ is a complete order isomorphism onto its range 
and a right inverse of $\psi_n$. 
Moreover, the map 
$\gamma_n\otimes\gamma_m : NC(n)\otimes_{\min} NC(m) \to \cl S_n\otimes_{\min} \cl S_m$
sis completely positive. 
A standard diagram chase now shows that \eqref{th_eqnc-2} holds: namely,
$NC(n)\otimes_{\min} NC(m) = NC(n)\otimes_{\max} NC(m)$. 
\end{proof}

We conclude this section with another realisation of $NC(n)$ as 
a quotient of a matrix operator system, which leads to a different characterisation of WEP. 
Following \cite{farenick--paulsen2011}, let 
$$\cl W_n = {\rm span}\{u_iu_j^* : i,j = 0,1,\dots,n\}\subseteq C^*(\bb{F}_n),$$
where we have set $u_0 = 1$. 
Let $\beta : M_{n+1}\to \cl W_n$ be the linear map given by 
$\beta(E_{i,j}) = \frac{1}{n+1}u_i^* u_j$, $i,j = 0,\dots,n$. 
It follows from \cite{farenick--paulsen2011} that 
$\beta$ is a complete quotient map with kernel 
the space $D_{n+1}^0$ of all diagonal matrices of trace zero; 
thus, $\beta : M_{n+1}/D_{n+1}^0 \to \cl W_n$ is a complete order isomorphism. 
(We note that the map sending $E_{i,j}$ to $u_i u_j^*$ was considered in 
\cite{farenick--paulsen2011} but since $\{u_1^*,\dots,u_n^*\}$ is a set of universal unitaries
whenever $\{u_1,\dots,u_n\}$ is such, the claims remain true 
with our definition as well.)
Clearly, $\cl S_n\subseteq \cl W_n$ and 
$$\cl R_{n+1} \stackrel{def}{=}\beta^{-1}(\cl S_n) = 
{\rm span}\{E_{1,j}, E_{j,1}, E_{j,j} : j = 1,\dots,n+1\}.$$
Let $\gamma$ denote the restriction of $\beta$ to $\cl R_{n+1}$. We claim that $\gamma$
is a complete quotient map from $\cl R_{n+1}$ onto $\cl S_n$. 
Indeed, if $X$ is a strongly positive element of $M_k(\cl S_n)$
then $X$ is also strongly positive as an element of $M_k(\cl W_n)$. 
By \cite[Proposition 3.2]{farenick--kavruk--paulsen2011}, 
there exists $Y\in M_k( M_{n+1})$ such that $\beta^{(k)}(Y) = X$.
However, $R\in M_k(\cl R_{n+1})$ by the definition of $\cl R_{n+1}$.

We note that the map $\gamma$ is defined by the relations 
$\gamma(E_{i,i}) = \frac{1}{n+1} 1$, $\gamma(E_{1,i}) = \frac{1}{n+1} u_i^*$, 
$i = 1,\dots,n+1$.

\begin{proposition}\label{p_anqr}
The map $\psi\circ \gamma : \cl R_{n+1}\to NC(n)$ is a complete 
quotient map with kernel 
$$\cl L_{n+1} = {\rm span}\left\{\sum_{i=1}^{n+1} a_i E_{i,i} + 
\sum_{j=2}^{n+1} b_j(E_{1,j} - E_{j,1}) : \sum_{i=1}^{n+1} a_i = 0\right\}.$$
\end{proposition}
\begin{proof}
Since both $\gamma : \cl R_{n+1}\to \cl S_n$ and $\psi : \cl S_n \to NC(n)$ 
are complete quotient maps, the map $\psi\circ\gamma$ is also 
a complete quotient. 
The identification of its kernel is straightforward. 
\end{proof}

Since the graph underlying the operator system $\cl R_{n+1}$ is 
chordal (in fact, it is a tree and hence does not have cycles), 
$\cl R_{n+1}\otimes_{\min}\cl A = \cl R_{n+1}\omax\cl A$ 
for any unital C*-algebra $\cl A$ 
(see \cite[Proposition 6.7]{kavruk--paulsen--todorov--tomforde2011}). Thus, 
a version of Theorem \ref{p_eqmma} can be formulated with 
$\cl R_{n+1}$ in the place of $\cl T_{n+1}$. 
The methods in the proof of Corollary \ref{c_nc3} can be used 
to obtain the following characterisation of WEP.

\begin{corollary}\label{c_weprn}
The following statements are equivalent for a unital C*-algebra $\cl A$:
\begin{enumerate}
\item\label{c_weprn-1}
$\cl A$ has WEP;
\item\label{c_weprn-2} whenever $A_0,A_1,A_2,A_3\in M_k(\cl A)$ are such that 
$1\otimes A_0 + h_1\otimes A_1 + h_2\otimes A_2 + h_3\otimes A_3$ is strongly positive in $NC(3)\otimes_{\min}M_k(\cl A)$, 
there exist elements $A,B,C,D,X,Y,Z\in M_k(\cl A)$ with 
$A+B+C + D = A_0$, $X + X^* = A_1$, $Y + Y^* = A_2$ and $Z + Z^* = A_3$ 
such that the matrix 
$$\left(\begin{matrix} 
A & X & Y & Z\\ 
X^* &  B & 0 & 0\\ 
Y^* & 0 & C & 0\\
Z^* & 0 & 0 & D
\end{matrix}\right)$$ is strongly positive in $M_{4k}(\cl A)$. 
\end{enumerate}
\end{corollary}

Although Corollary~\ref{c_nc3} and Corollary~\ref{c_weprn} both give characterisations of WEP in terms of ``completions'' of $4 \times 4$ matrices, there does not 
appear to be a direct connection between the two sets of conditions. In fact, even though both of these results arise from realising $NC(3)$ as a quotient of $\cl T_4$ and $\cl R_4$, 
respectively, we shall now show that these later operator systems are not completely order isomorphic.

\begin{proposition}\label{p_gr}
The operator systems $\cl R_n$ and $\cl T_n$ are not completely order 
isomorphic unless $n \in \{1,2,3\}$.
\end{proposition}
\begin{proof}
For a graph $\cl G$ on $n$ vertices, let $\cl S_{\cl G}$ be the 
\lq\lq graph operator system'' (see \cite{kavruk--paulsen--todorov--tomforde2011})
$$\cl S_{\cl G} = {\rm span}\{E_{i,j}, E_{k,k} : k = 1,\dots,n, (i,j)\in \cl G\}.$$

We first claim that if 
$\cl G_1$ and $\cl G_2$ are connected graphs on $n$ vertices and 
$\nph : \cl S_{\cl G_1}\to \cl S_{\cl G_2}$ is a 
complete order isomorphism then there exists a unitary
$U\in M_n$ such that $U^*\cl S_{\cl G_1} U = \cl S_{\cl G_2}$.
Indeed, since $\cl G_1$ and $\cl G_2$ are connected, the C*-algebras
$C^*(\cl S_{\cl G_1})$ and $C^*(\cl S_{\cl G_2})$ generated by 
$\cl S_{\cl G_1}$ and $\cl S_{\cl G_2}$, respectively, both coincide with $M_n$. 
Since $M_n$ is simple, we have that 
the C*-envelopes 
$C_{\rm e}^*(\cl S_{\cl G_1})$ and $C_{\rm e}^*(\cl S_{\cl G_2})$ of
$\cl S_{\cl G_1}$ and $\cl S_{\cl G_2}$, respectively, both coincide with $M_n$.
The complete order isomorphism $\nph$ now gives rise to 
an isomorphism between their C*-envelopes and hence there exists an isomorphsim
$\tilde{\nph} : M_n\to M_n$ extending $\nph$. 
Let $U\in M_n$ be a unitary matrix with 
$\tilde{\nph}(A) = U^*AU$, $A\in M_n$; then 
$U^*\cl S_{\cl G_1} U = \cl S_{\cl G_2}$.

Now note that $\cl T_{n+1}$ and $\cl R_{n+1}$ are both graph operator systems. 
Let $P_k = U^*E_{k,k}U$, $k = 1,\dots,n+1$,
and $\cl C = {\rm span}\{P_k : k = 1,\dots,n+1\}$.
Since $\cl T_{n+1}$ is a bimodule over the algebra $D_{n+1}$
of all diagonal matrices, $\cl R_{n+1}$ is a bimodule 
over $\cl C$. Note that each $P_{k}$ is a rank one operator. 
Assume that not all of $P_1,\dots,P_{n+1}$ are equal to a diagonal 
matrix unit in $\cl R_{n+1}$, suppose, for example, that 
$P_1 = (\lambda_i\overline{\lambda_j})_{i,j=1}^{n+1}$ is not 
of the form $E_{k,k}$. 
Set $\Lambda = \{k : \lambda_k\neq 0\}$; then 
${\rm span}\{E_{i,j} : i,j\in \Lambda\}\subseteq \cl R_{n+1}$.
However, the only full matrix subalgebras of $\cl R_{n+1}$ are of the form 
${\rm span}\{E_{1,1}, E_{1,j}, E_{j,1}, E_{j,j}\}$, for some $j$. 
Assume, without loss of generality, that $j =1$. 
But then $P_1 E_{1,3}$ has $\lambda_2\overline{\lambda_1}$ as its 
$(2,3)$-entry, contradicting the definition of $\cl R_{n+1}$. 

It follows that $\{P_k\}_{k=1}^{n+1} = \{E_{k,k}\}_{k=1}^{n+1}$, so that 
there exists a permutation $\pi$ of $\{1,\dots,n+1\}$ with 
$P_k = E_{\pi(k),\pi(k)}$, $k = 1,\dots,n+1$. If we let $U_{\pi}$ denote the
corresponding permutation unitary, then $U_{\pi}E_{k,k}U_{\pi}^* = P_k = U^*E_{k,k}U.$ Hence,
$U_{\pi}^* U^*E_{k,k}UU_{\pi} = E_{k,k}$ for all $k$ and consequently, $UU_{\pi}$ is diagonal.  Thus, $U_{\pi} \cl T_{n+1} U_{\pi}^* = U^* \cl T_{n+1} U = \cl R_{n+1}.$ 
This means that  $\pi$ defines an isomorphism of the 
underlying graphs of $\cl T_{n+1}$ and $\cl R_{n+1}.$  This is a contradiction if $n \geq 3$ since the graph underlying $\cl T_{n+1}$
has at least two vertices of degree $2$, while the graph underlying $\cl R_{n+1}$
has only one vertex of degree bigger than $1$.
\end{proof}


\section{$NC(n)$ as a quotient of $\bb{C}^{2n}$}\label{s_cop}

In this section we represent $NC(n)$ as an operator system 
quotient of the abelian C*-algebra $\bb{C}^{2n}$ in two different ways and include some 
consequences of these results. In the next section we will use these two representations to give two more characterisations of WEP.
We first recall some 
basic facts about coproducts of operator systems.
Coproducts in this category were used by D. Kerr and H. Li \cite{kerrli}, 
where the authors described
the amalgamation process over a joint operator subsystem. 
T. Fritz  demonstrated some applications of this concept
in quantum information theory \cite{fritz0}. A categorical 
treatment and further results can be found in the thesis of the second author \cite{kavruk2011}. 
We next extend the results from \cite{fritz0} and \cite{kavruk2011}
to deduce representations of the coproduct
of (finitely) many operator systems.

Let $\cl S_1,\dots,\cl S_n$ be operator systems. 
Then there exists a unique operator system $\cl U$, along with the unital
complete order embeddings $i_m: \cl S_m \hookrightarrow \cl U$, $m=1,\dots,n$, 
such that the following universal
property holds: For any operator system $\cl T$ and 
unital completely positive maps 
$\varphi_m: \cl S_m \rightarrow \cl T$, $m=1,\dots,n$,
there is a unique unital completely positive map 
$\varphi: \cl U \rightarrow \cl T$ such that $\varphi_m = \varphi \circ i_m$ for all $m$.
In fact, using F. Boca's results \cite{Boca91}, it can be easily 
shown that the operator system
$$
{\rm span} \{ s_1 + \cdots + s_n: \; s_m \in \cl S_m,\; m=1,\dots,n \} \subseteq C^*_u(\cl S_1)*\cdots  *C^*_u(\cl S_n)
$$
satisfies this condition, while its uniqueness is a standard consequence 
of its universal property.
The operator system $\cl U$ will be called the 
\emph{coproduct} of $\cl S_1,\dots,\cl S_n$ and denoted by
$\amalg_{m=1}^n \cl S_m.$
We often identify each $\cl S_m$ with its canonical image $i_m(\cl S_m)$
in $\amalg_{m=1}^n \cl S_m.$ 

As in \cite{kavruk2011}, a more concrete realisation of the coproduct can be given in terms of operator system quotients by null subspaces:

\begin{theorem}\label{thm coproduct}
Let $\cl S_1,\dots,\cl S_n$ be operator systems and 
$$\cl J = {\rm span}\{ (e,-e,0,\dots,0), (e,0,-e,0,\dots,0), ..., (e,0,\dots,0,-e) \}.$$
Then $\cl J$ is a kernel and, up to a 
unital complete order isomorphism, 
$$\amalg_{m=1}^n \cl S_m \cong (\cl S_1 \oplus \cdots \oplus \cl S_n) / \cl J.$$
\end{theorem}
\begin{proof}
The fact that $\cl J$ is a null subspace (and hence a kernel) is straightforward.
Note that
$$
(e,\dots,e) + \cl J = (ne,0,\dots,0) + J = \cdots  = (0,\dots,0,ne) + \cl J.
$$
Therefore, the map
$\iota_m : \cl S_m \rightarrow (\oplus_{j=1}^n \cl S_j)/\cl J$ given by $s\mapsto (0,\dots,n s,\dots,0) + \cl J$, where the term $ns$ appears at the $m^{th}$-component, 
is unital and completely positive. 

We claim that
$\iota_m$ is a complete order embedding. To prove this, first note that
$$\cl J = {\rm span}\{x_j: j\in \{1,\dots,n\}\setminus\{m\}\},$$
where $x_j$ has the unit $e$ as its $m^{th}$-component, $-e$ as its
$j^{th}$-component and $0$'s elsewhere
(we leave the elementary verification of this to the reader). 

Suppose that 
$\iota_m(s)$ is positive in $(\oplus_{j=1}^n \cl S_j)/\cl J$. 
We will prove that $s$ is positive in $\cl S_m$.
Since null subspaces are (completely) proximinal kernels \cite[Proposition  2.4]{kavruk2011},
it follows that $\iota_m(s)$ has a positive lifting 
$$
y = (0,...,ns,...,0) + \sum_{j \neq m} a_j x_j
$$
in $\oplus_{j=1}^n \cl S_j$.
Using the definition of $x_j$, we see that
$$
y = (-a_1 e,\dots,-a_{m-1}e,ns + \Sigma_{j\neq m} a_j e,-a_{m+1}e,\dots,-a_ne).
$$
Now it is clear that $a_j \leq 0$, $j \neq m$, and hence $n s \geq - \Sigma_{j\neq m} a_m e \geq 0$.
This proves that $\iota_m$ is an order embedding. A similar argument shows that 
$\iota_m$ is a complete order embedding.

As a second step, we show 
that $(\oplus_{j=1}^n \cl S_j)/\cl J$ has the universal property of the coproduct. 
Let $\cl T$
be an operator system and $\varphi_m: \cl S_m \rightarrow \cl T$ be 
a unital completely positive map, $m = 1,\dots,n$. 
Let $\tilde \varphi : \oplus_{j=1}^n \cl S_j \rightarrow \cl T$ be given by
$\tilde \varphi(s_1,\dots,s_n) = (1/n) \sum_{j=1}^n \varphi_j(s_j)$. Then 
$\cl J\subseteq \ker\tilde \varphi$ and hence there exists a unital completely 
positive map
$\varphi: (\oplus_{j=1}^n \cl S_i)/\cl J \rightarrow \cl T$ such that
$$
\nph((s_1,\dots,s_n) + \cl J) = \frac{1}{n} \sum_{j=1}^n \varphi_j(s_j).
$$
It is now elementary to see that 
$\varphi_m =  \varphi \circ \iota_m$ for every $m$. 
Since coproducts are unique up to a complete order isomorphism, the result follows.
\end{proof}

It is easy to verify that coproducts satisfy the associative law. 
The universal property of coproducts ensures that, for the operator systems $\cl S_1,\dots,\cl S_n$, 
there is a canonical C*-algebraic identification
$$
C^*_u(\amalg_{i=1}^n \cl S_i)  \cong  C^*_u(\cl S_1)*\cdots  *C^*_u(\cl S_n),
$$
where $*$ denotes free product amalgamated over the unit. 
In fact, we have the following stronger result.


\begin{theorem}\label{th_sg}
Let $\cl S_i$ be an operator subsystem of a C*-algebra $\cl A_i$, $i=1,\dots,n$. Let
$$
\cl S \stackrel{def}{=} {\rm span}\{s_1 + \cdots + s_n: \; s_i \in \cl S_i,\; i=1,\dots,n\} 
\subseteq \cl A_1 * \cdots * \cl A_n.
$$
Then the canonical map $\amalg_{j=1}^n \cl S_j \rightarrow *_{j=1}^n \cl A_j$ associated with the inclusions
$i_m: \cl S_m \hookrightarrow \cl A_m$, $m = 1,\dots,n$, is a unital complete order embedding with image $\cl S$.
If, moreover, each
$\cl S_i$ is spanned by unitaries that generate $\cl A_i$ as a C*-algebra, 
then $\amalg_{i=1}^n \cl S_i$ is spanned by unitaries that generate
$\cl A_1*\cdots * \cl A_n$ as a C*-algebra, and
$$
C^*_e(\amalg_{i=1}^n \cl S_i) \cong \cl A_1*\cdots * \cl A_n.
$$
\end{theorem}

\begin{proof}
We shall prove that $\cl S$ has the desired universal property. 
Let $\cl T\subseteq \cl B(H)$ be an operator system and
$\varphi_m: \cl S_m \rightarrow \cl T$
be a unital completely positive map, $m = 1,\dots,n$.
Let $\tilde \varphi_m: \cl A_m\rightarrow \cl B(H)$
be a unital completely positive extension of $\varphi$ 
and let $\varphi: \cl A_1*\cdots * \cl A_n \rightarrow \cl B(H)$
be the unital completely positive map arising from Theorem \ref{th_boca}. 
Clearly, $\varphi|_{\cl S}$ has the desired properties and has image inside $\cl T$.

Suppose that each $\cl A_m$ is generated by 
a family of unitaries in $\cl S_m$. Since
the free product $\cl A_1*\cdots * \cl A_n$ 
is generated by $\cl A_1 \cup \cdots\cup \cl A_n$, 
it is generated by the unitaries in $\amalg_{i=1}^n \cl S_i$.
The remaining part of the theorem is a direct consequence of \cite[Proposition 5.6]{kavruk2011}.
\end{proof}

\begin{corollary}\label{c_u}
Let $G_i$ be a discrete group generated by the set $\frak{u}_i$, for each $i = 1,\dots,n$.
Set $\frak u = \frak{u}_1 \cup \cdots \cup \frak{u}_n$, 
viewed as a generating set for $G_1\ast\cdots\ast G_n$. Then 
$
\cl S (\frak{u}) = \amalg_{i=1}^n \cl S(\frak{u}_i).
$
\end{corollary}
\begin{proof}
The claims follow from Theorem \ref{th_sg}. 
\end{proof}

Since the characters ({\it i.e.}, 
the one dimensional unitary representations) of $\bb Z_k$ can be identified with
distinct $k^{th}$-roots of unity, the Fourier transform gives a C*-algebraic  
identification $C^*(\bb Z_k) \cong \bb C^k$; see, {\it e.g.},  \cite[Section VII.1]{Davidson}.

\begin{corollary}\label{cor NPOL}
We have a complete order isomorphism
$$
\amalg_{i=1}^n \bb C^k \cong {\rm span} \{ e,a_1^i,a_2^i,\dots,a_n^{i}: i = 1,\dots,k-1\} \subseteq 
C^*(\bb Z_k * \cdots  *\bb Z_k)
$$
where the free product consists of $n$ terms and 
$a^i_j$, $j = 1,\dots,k$, are the canonical generators of $C^*(\bb Z_k)$.
In particular, for every $n$,
\begin{equation}\label{eq_ncn}
\amalg_{i=1}^n \bb C^2 \cong NC(n) \mbox{ unitally and completely order isomorphically}.
\end{equation}
\end{corollary}
\begin{proof} The assertion
follows from Corollary \ref{c_u} by setting each $G_i = \bb Z_k$.
\end{proof}

As the unital complete order isomorphism in Corollary \ref{cor NPOL} is based on
the Fourier transform, it will be convenient to have a more concrete realisation.
Consider $NC(n)$ with its standard basis $\{ e,h_1,\dots,h_n \}$, where $h_k$ is a 
universal selfadjoint contraction, $k = 1,\dots,n$. 
Let 
$$
p_k = \frac{1+h_k}{2}, \mbox{ so that } p_k^{\perp} = \frac{1-h_k}{2}, \;\; k=1,\dots,n.
$$ 
As usual, $\{e_k\}_{k=1}^{2n}$ denotes the standard basis of $\bb C^{2n}$.

\begin{theorem}\label{th_anoq}
Let $\theta : \bb C^{2n} \to NC(n) $ be the linear map given by
$$\theta(e_{2k-1}) = \frac{1}{n} p_k, \ \theta(e_{2k}) = \frac{1}{n} p_k^{\perp}, \ k = 1,\dots,n.$$
Then $\theta$ is a completely positive complete quotient map onto $NC(n)$ with kernel 
$$\cl J_n = {\rm span}\{(e,-e,0,\dots,0),(e,0,-e,\dots,0),\dots,(e,0,0,\dots,-e)\}.$$
\end{theorem}
\begin{proof}
Write $i_k$ for the canonical inclusion of the $k$th copy of $\bb{C}^2$ into 
$\amalg_{j=1}^n \bb C^2$, and
$\{f_1,f_2\}$ for the standard basis of $\bb{C}^2$.
By (\ref{eq_ncn}), we have a canonical complete order unital 
isomorphism $NC(n) \cong \amalg_{j=1}^n \bb C^2$, 
under which the element $h_k\in NC(n)$ is identified 
with the element $i_k(f_1 - f_2)\in \amalg_{j=1}^n \bb C^2$.

On the other hand, by Theorem \ref{thm coproduct}, 
there is a unital complete order isomorphism
$\nph : \amalg_{j=1}^n \bb C^2 \to \bb{C}^{2n}/\cl J_n$, such that 
$\nph \circ i_k = \iota_k$ (where $\iota_k$ is the 
order embedding defined in Theorem \ref{thm coproduct}). 
Now note that 
$$\nph(h_k) = \nph(i_k(f_1 - f_2)) = n(e_{2k-1} - e_{2k}) + \cl J_n.$$
On the other hand, 
$$\nph(1) = \nph(i_k(f_1 + f_2)) = n(e_{2k-1} + e_{2k}) + \cl J_n.$$
It follows that 
$$\nph(p_k) = n e_{2k-1} \ \ \mbox{ and } \ \ \nph(p_k^{\perp}) = n e_{2k}.$$
Thus, the map $\theta$ is the inverse of $\nph$ and the proof is complete. 
\end{proof}


It is now possible to give another proof of Corollary~\ref{p_llp}.
Since $NC(n) = \amalg_{i=1}^n \bb C^2 = \bb C^{2n}/\cl J$, where $\cl J$ is the 
null subspace as in Theorem \ref{thm coproduct}, and the
lifting property is preserved under quotients by null subspaces 
\cite[Theorem 6.9]{kavruk2011}, the lifting result follows. The tensor 
identity follows from
the local lifting property criteria given in 
\cite[Theorem 8.5]{kavruk--paulsen--todorov--tomforde2010}.

Our next aim is to represent $NC(n)$ as a quotient of $\bb{C}^{2n}$ 
in a different way than the one exhibited in Theorem \ref{th_anoq}. 
We first note a general fact about quotients.

\begin{proposition}\label{p_qq}
Let $\cl S$ be a finite dimensional operator system, $\cl J_1\subseteq \cl S$ be a null space,
$q : \cl S\to \cl S/\cl J_1$ be the quotient map and $\cl J_2\subseteq \cl S/\cl J_1$ be a null space. 
Then $q^{-1}(\cl J_2)$ is a null space in $\cl S$ and 
$(\cl S/\cl J_1)/\cl J_2$ is canonically completely order isomorphic to $\cl S/q^{-1}(\cl J_2)$. 
\end{proposition}
\begin{proof}
Note that $q^{-1}(\cl J_2)$ is finite dimensional as both $\cl J_1$ and $\cl J_2$ are. Let $y\in q^{-1}(\cl J_2)$ 
be such that $y\in \cl S_+$. Then $q(y) \in \cl J_2$ and $q(y) \in (\cl S/\cl J_1)_+$; 
since $\cl J_2$ is a null space, $q(y) = 0$, or
$y\in \cl J_1$. Since $\cl J_1$ is a null space, we have that $y = 0$. 

Since $\cl J_1\subseteq q^{-1}(\cl J_2)$, there exists a unital completely positive map $\phi : \cl S/\cl J_1\to \cl S/q^{-1}(\cl J_2)$
given by $\phi(x +\cl  J_1) = x + q^{-1}(\cl J_2)$, $x\in \cl S$. 
If $q(x) = x + \cl J_1\in \cl J_2$ then $x\in q^{-1}(\cl J_2)$; thus, $\cl J_2\subseteq \ker\phi$. Let 
$\psi : (\cl S/\cl J_1)/\cl J_2\to \cl S/ q^{-1}(\cl J_2)$ be the induced untal completely positive map. 
The map $\psi$ is bijective, and it remains to show that $\psi^{-1}$ is completely positive. To this end, 
let $(x_{ij} + q^{-1}(\cl J_2))_{i,j}\in M_n(\cl S/q^{-1}(\cl J_2))_+$. This means that for every $\epsilon > 0$ there exist
$y_{ij}\in q^{-1}(\cl J_2)$ such that $(x_{ij} + y_{ij})_{i,j} + \epsilon 1_n \in M_n(\cl S)_+$. 
It follows that $(x_{ij} + y_{ij} + \cl J_1)_{i,j} + \epsilon 1_n \in M_n(\cl S/\cl J_1)_+$, which shows that $\psi$ is a complete order isomorphism. 
\end{proof}

The following is easy to verify. 

\begin{lemma}\label{l_su}
Let $\cl S$ and $\cl T$ be operator systems and $\cl J\subseteq \cl S$ be a kernel. 
Then $\cl J\oplus 0$ is a kernel in $\cl S\oplus \cl T$ and,
up to a complete order isomorphism,
$(\cl S/\cl J)\oplus \cl T = (\cl S\oplus \cl T)/(\cl J\oplus 0)$.
\end{lemma}

\begin{theorem}\label{th_nc}
Let $e$ be the unit of $\bb{C}^2$ and, for $n\in \bb{N}$, let
$$\cl Q_n = {\rm span}\{(e,-e,0,0,\dots),(e,e,-e,0,\dots),\dots,(e,\dots,e,-e)\}\subseteq \bb{C}^{2n}.$$
Then $\cl Q_n$ is a null subspace of $\bb{C}^{2n}$ and 
$NC(n)$ is completely order isomorphic to $\bb{C}^{2n}/\cl Q_n$.
\end{theorem}
\begin{proof}
Let $\tilde{e}$ be the unit of $\bb{C}^{2n-2}/\cl Q_{n-1}$. We use induction. 
The case $n = 2$ is in \cite{kavruk2012}. 
Assuming $\cl Q_{n-1}$ is a null subspace of $\bb{C}^{2n-2}$, we have that 
$\cl Q_{n-1}\oplus 0$ is a null subspace of $\bb{C}^{2n}$. 
Let $q : \bb{C}^{2n} \to \bb{C}^{2n} / (\cl Q_{n-1}\oplus 0)$ be the quotient map. 
Then $I \stackrel{def}{=} {\rm span}\{(e,\dots,e,-e) + (\cl Q_{n-1}\oplus 0)\}$ is easily seen to be a null space in $\bb{C}^{2n} / (\cl Q_{n-1}\oplus 0)$ and 
$\cl Q_n = q^{-1}(I)$; by Proposition \ref{p_qq}, $\cl Q_n$ is a null space in $\bb{C}^{2n}$.

Now, using successively the definition of $NC(n)$, the induction hypothesis, \cite[Proposition 4.7]{kavruk2012}, 
Lemma \ref{l_su} and Proposition \ref{p_qq}, we have 
\begin{eqnarray*}
NC(n) & = & NC(n-1) \amalg \bb{C}^2 = (\bb{C}^{2n-2}/\cl Q_{n-1}) \amalg \bb{C}^2\\
& = & 
((\bb{C}^{2n-2}/\cl Q_{n-1})\oplus \bb{C}^2)/ {\rm span}\{(\tilde{e},-e)\}\\ & = & 
\left(\bb{C}^{2n}/(\cl Q_{n-1}\oplus 0)\right)/  I = 
\bb{C}^{2n}/\cl Q_n.
\end{eqnarray*}
\end{proof}

We note that the kernels $\cl J_n$ and $\cl Q_n$ are different, hence Theorems 
\ref{th_anoq} and \ref{th_nc} provide two distinct realisations of 
$NC(n)$ as a quotient of $\bb{C}^{2n}$.

By a matrix 
operator system we shall mean an operator subsystem $\cl S$ of a matrix algebra.
It was observed in \cite[Proposition 5.13]{fkpt} that the non-commutative cube
$NC(n)$ 
can be realised as the dual of a matrix operator system.
We next provide a multivariable version of 
\cite[Proposition 5.11]{fkpt}; that latter result
identified $NC(2)$ with a dual of a {\it diagonal} matrix operator system. 
This result can also be found in Ozawa \cite{ozawa2013}.

\begin{theorem}
Let
$$
\cl R_{n,k}  = \{ (a_1^1,\dots,a_k^1,\dots,a_1^n,\dots,a_k^n) \in \bb{C}^{nk} : 
\sum_{i=1}^k a_i^l = \sum_{i=1}^k a_i^m, \mbox{ for all } l,m   \}.
$$
Then $\cl R_{n,k}^* \cong \amalg_{i=1}^n \bb C^k$ unitally and completely order isomorphically.
In particular, if
$$
\cl R = \{ (a_1,a_2,\dots,a_{2n-1},a_{2n})\in \bb{C}^{2n} : a_1 + a_2 = \cdots = a_{2n-1} + a_{2n}\} 
$$
then $\cl R^* = NC(n)$ unitally and completely order isomorphically. 
\end{theorem}

\begin{proof}
Since
$\amalg_{i=1}^n \bb C^k = (\bb C^k \oplus \cdots \oplus \bb C^k) / \cl J \cong  \bb C^{nk}/J$ where 
$\cl J$ is the null subspace defined in Theorem \ref{thm coproduct}, 
and, by \cite[Proposition 2.7]{kavruk2011}, the adjoint of a complete 
quotient map is a complete order embedding, it follows that
$$
(\amalg_{i=1}^n \bb C^k)^* \hookrightarrow (\bb C^{nk})^* \cong \bb C^{nk}.
$$
Clearly the image of this map is precisely the subspace described in the theorem. Moreover,
by defining the Archimedean order unit of $(\bb C^{2n})^*$ to be the positive linear functional  which 
maps each element to the sum of its entries, it is
elementary to see that the identification is unital.  
\end{proof}


\section{The Riesz decomposition property}\label{s_rdp}

In \cite{kavruk2012} the second author characterised WEP in terms of a relative non-commutative Riesz interpolation property. 
In this section we characterise WEP in terms of 
a relative non-commutative Riesz decomposition property. 
Recall that an element $s$ in an operator system is called strongly
positive, denoted by $s \gg 0$, if $s \geq \delta e$ for some $\delta>0$. 
By $s \gg t$, we mean $s - t \gg 0$, and by $s_1,s_2 \gg t$, we mean 
$s_1 \gg t$ and $s_2 \gg t$.


\begin{definition}\label{d_rds}
Let $\cl B$ be a unital C*-algebra. 

(i) A ordered tuple $(a_1,a_2,b_1,\dots,b_n,x_1,\dots,x_n)$ of 
elements of $\cl B$, where $x_1,\dots,x_n$ are strongly positive, 
will be called a \emph{Riesz decomposition scheme}
if 
\begin{enumerate}
 \item $a_1 \gg x_1 + \cdots + x_n$
 \item $a_2 \gg x_1 + \cdots + x_n$
 \item $x_i \gg b_i$, $i=1,\dots,n$.
\end{enumerate}

(ii) Let $\cl A$ be a unital C*-subalgebra of $\cl B$. We say that $\cl A$ has 
the \emph{$n$-Riesz decomposition property in $\cl B$}
if, whenever 
$(a_1,a_2,b_1,\dots,b_n,x_1,\dots,x_n)$ is a Riesz decomposition scheme, 
where $a_1,a_2,b_1,\dots,b_n \in \cl A$ and $x_1,\dots,x_n \in \cl B$, then there
exist strongly positive elements
$y_1,\dots,y_n\in \cl A$ 
such that 
$(a_1,a_2,b_1,\dots,b_n,$ $y_1,\dots,y_n)$ is a Riesz decomposition scheme.

We say that $\cl A$ has the \emph{complete $n$-Riesz decomposition property in $\cl B$} 
if $M_k(\cl A)$ has the $n$-Riesz decomposition property in $M_k(\cl B)$ for every $k$.
\end{definition}

Although the definition is given for an arbitrary pair of C*-algebras $\cl A$ and $\cl B$ 
with $\cl A \subseteq \cl B$, 
we will be mostly concerned with the case where $\cl B$ is injective. In fact, injective objects
are universal in the sense that if $\cl A$ is represented as a unital C*-subalgebra of 
both $\cl B_1$ and $\cl B_2$, where $\cl B_1$ and $\cl B_2$ are injective, 
and $\cl A$ has the complete $n$-Riesz decomposition property in $\cl B_1,$ 
then $\cl A$ has the complete $n$-Riesz decomposition in $\cl B_2.$
This follows from a straightforward application
of Arveson's extension theorem. In particular, we see that $\cl A$ has the complete $n$-Riesz decomposition property in an injective C*-algebra $\cl B$ if and only if $\cl A$ has the complete $n$-Riesz decomposition in $I(\cl A).$

In the sequel, 
the non-commutative cubes will be identified with the matrix quotients via 
Theorem \ref{th_anoq}:
\begin{equation}\label{eq_ida}
NC(n) = \amalg_{i=1}^n \bb C^2 \cong \bb C^{2n} / \cl J_n
\end{equation}
where $\cl J_n$ denote the $(n-1)$-dimensional null-subspace
$$
{\rm span}\{ (1,1,-1,-1,0,...,0), (1,1,0,0,-1,-1,0,...,0),...,(1,1,0,...,0,-1,-1)\}.
$$
If $x\in \bb{C}^{2n}$, we let $\dot{x} = x + \cl J_n$ be the image of $x$ under 
the quotient map; note that via the identification (\ref{eq_ida}), 
$\dot{x}$ can be viewed as an element of $NC(n)$. 

We recall that $\{e_i\}_{i=1}^{2n}$ is the standard basis of $\bb{C}^{2n}$.
It is elementary to verify that $\{\dot{e}_1,\dot{e}_2,\dot{e}_3, \dot{e}_5 , ... ,\dot{e}_{2n-1}\}$ 
is a basis for $NC(n)$.
We will need the following positivity criteria for elements in an operator system 
of the form
$\cl S\otimes_{\max} NC(n)$.

\begin{proposition}\label{prop pos-crit}
Let $\cl S$ be an operator system and 
$$
u = s_1 \otimes \dot{e}_1 + s_2 \otimes \dot{e}_2 - s_3 \otimes \dot{e}_3 - s_5 \otimes \dot{e}_5 - \cdots - s_{2n-1} \otimes \dot{e}_{2n-1} \in \cl S \otimes NC(n).
$$ 
Then $u \gg 0$ in $\cl S \otimes_{\max} NC(n)$
if and only if there are strongly positive elements $x_3,x_5,\dots,x_{2n-1}$ in $\cl S$ such that
$(s_1,s_2,s_3,s_5\dots,s_{2n-1},x_3,x_5,\dots,x_{2n-1})$ is a Riesz decomposition 
scheme. 
\end{proposition}
\begin{proof}
Denote by $q$ the quotient map from $\bb{C}^{2n}$ onto $NC(n) = \bb{C}^{2n}/\cl J_n$ and
suppose that
$u \gg 0$ in $\cl S \otimes_{\max} (\bb C^{2n}/\cl J_n)$. 
By \cite[Proposition 1.6]{farenick--paulsen2011}, 
$\id \otimes q :\cl S \otimes_{\max} \bb C^{2n} \rightarrow \cl S \otimes_{\max} (\bb C^{2n}/\cl J_n)$
is a complete quotient map. Thus, by Proposition~\ref{cqchar} 
$u$
lifts to a strongly positive
element of $\cl S \otimes_{\max} \bb C^{2n}$. 
Since 
$$\dot{e} = (n,n,0,\dots,0) + \cl J_n = n \dot{e}_1 + n \dot{e}_2,$$
we have that
$$
u  = s_1  \otimes \dot{e}_1 + s_2  \otimes \dot{e}_2 - \sum_{i=1}^{n-1} s_{2i+1} \otimes \dot{e}_{2i+1}.
$$
Taking into account that
$$
\cl J_n = {\rm span}\{ e_1 + e_2 - e_{2i+1} - e_{2i+2}: \; i = 1,\dots,n-1\},
$$
we see that a strongly positive lifting of 
$u $ in $\cl S \otimes_{\max} \bb C^{2n}$ has the form
$$
s_1  \otimes e_1 + s_2  \otimes e_2 - \sum_{i=1}^{n-1} s_{2i+1} \otimes e_{2i+1}$$ 
$$ + \sum_{i = 1}^{n-1} (-x_{2i+1}) \otimes (e_1 + e_2 - e_{2i+1} - e_{2i+2}) 
$$
for some $x_3,x_5,\dots,x_{2n-1}$ in $\cl S$. 
Since 
$\cl S \otimes_{\max} \bb C^{2n} = \cl S \otimes_{\min} \bb C^{2n}$, 
we deduce
$$
\left. \begin{array}{r}
 s_1  - x_3 - \cdots - x_{2n-1} \gg 0 \\
 s_2  - x_3 - \cdots - x_{2n-1} \gg 0 \\
  - s_3 + x_3 \gg 0 \\
          x_3 \gg 0 \\
    - s_5 + x_5 \gg 0 \\
             x_5 \gg 0 \\
         \vdots \;\; \\
    - s_{2n-1} + x_{2n-1} \gg 0 \\
             x_{2n-1} \gg 0 \\    
\end{array} \right\} \Longrightarrow
\begin{array}{l}
 s_1 \gg  x_3 + \cdots + x_{2n-1} \\
 s_2 \gg  x_3 + \cdots + x_{2n-1} \\
 x_3,x_5,\dots,x_{2n-1} \gg 0 \\
 s_{2i+1} \ll x_{2i+1},\;\; i = 1,\dots,n-1.
\end{array}
$$
and we have obtained our decomposition.

Conversely, whenever $s_1,s_2,s_3,s_5,\dots,s_{2n-1}$ have 
the property that $s_1,s_2 \gg x_3+\cdots + x_{2n-1}$, $s_{2i+1} \ll x_{2i+1}$; $i=1,2,\dots,n-1$,
for some strongly positive elements $x_3,\dots,x_{2n-1}$, by reversing the argument
in the previous paragraph, we deduce that $u$ is strongly positive.
\end{proof}

\begin{proposition}\label{p_1R}
Let $\cl A$ and $\cl B$ be unital C*-algebras with $\cl A \subseteq \cl B$. 
Then $\cl A$ has the $1$-Riesz decomposition
property in $\cl B$. In other words, for any $a_1,a_2,a_3\in \cl A$, 
whenever there is $x$ strongly positive in $\cl B$ 
satisfying $a_1,a_2 \gg x$, $a_3 \ll x$ then this $x$ can be chosen to be a strongly positive element of $\cl A$.
\end{proposition}
\begin{proof}
Suppose that 
$a_1,a_2,a_3\in \cl A$ and $x\in \cl B_+$ are such that
$a_1,a_2 \gg x$, $a_3 \ll x$.
By Proposition \ref{prop pos-crit}, 
$a_1 \otimes \dot{e}_1 + a_2 \otimes \dot{e}_2 - a_3 \otimes \dot{e}_3$ is strongly
positive in $\cl B \otimes_{\max} NC(2)$. 
By \cite[Proposition 6.3]{fkpt}, $NC(2)$ is C*-nuclear, and hence 
$a_1 \otimes \dot{e}_1 + a_2 \otimes \dot{e}_2 - a_3 \otimes \dot{e}_3$ is strongly
positive in $\cl B \otimes_{\min} NC(2)$.
The injectivity of the minimal tensor product now 
ensures that $a_1 \otimes \dot{e}_1 + a_2 \otimes \dot{e}_2 - a_3 \otimes \dot{e}_3$
is strongly positive in $\cl A \otimes_{\min} NC(2)$. 
Another application of 
the C*-nuclearity of $NC(2)$ and Proposition \ref{prop pos-crit}
establishes the claim.
\end{proof}

While the $n$-Riesz decomposition is automatically satisfied if $n = 1$, 
higher values of $n$ require, as shown in the following theorem, 
an additional assumption on $\cl A$.

\begin{theorem}\label{thm WEPRiesz}
For a unital C*-subalgebra  $\cl A \subseteq  \cl B(H)$, the following statements are equivalent:
\begin{enumerate}
 \item[(i)] $\cl A$ has WEP;
 \item[(ii)] $\cl A$ has the complete $2$-Riesz decomposition property in $\cl B(H)$;
 \item[(iii)] $\cl A$ has the complete $n$-Riesz decomposition property in $\cl B(H)$ for every $n\in \bb{N}$.
\end{enumerate}
\end{theorem}

In contrast to the original definition of WEP given in \cite{Lance}, 
the characterisation given in Theorem \ref{thm WEPRiesz} 
only makes reference to a single concrete representation 
of $\cl A$. Moreover, we shall see below that 
$\cl B(H)$ can be replaced by an arbitrary C*-algebra having WEP. 
The proof of the theorem will be based on the following result:

\begin{proposition}\label{prop RD&T}
Let $\cl B$ be a unital C*-algebra and $\cl A\subseteq \cl B$ be a unital C*-subalgebra. 
The following are equivalent:
\begin{enumerate}
 \item[(i)] $\cl A$ has the (complete) $n-1$-Riesz decomposition property in $\cl B$;
 \item[(ii)] There is a canonical (complete) order embedding
$$
\cl A \otimes_{\max} NC(n) \subseteq_{{\rm coi}} \cl B \otimes_{\max} NC(n).
$$
\end{enumerate}
\end{proposition}
\begin{proof}
We first skip ``complete'' and prove the equivalence of (i) and (ii). 
Since $M_k(\cl A)$ is a unital C*-subalgebra of $M_k(\cl B)$
and 
$$M_k(\cl C \otimes_{\max} \cl T) = M_k(\cl C) \otimes_{\max} \cl T$$ canonically 
for every C*-algebra $\cl C$ and any operator system $\cl T$,
the equivalence of the statements with the term \lq\lq complete'' added 
will be automatically satisfied.

We identify $NC(n)$ with $\bb C^{2n}/\cl J_{n}$ and fix the basis $\{y_1,y_2,x_1,\dots,x_{n-1}\}$ where
$y_1 = \dot{e}_1,$ $ y_2= \dot{e}_2$ and $x_{i} = \dot{e}_{2i+1}$, $i = 1,\dots,n-1$.

(i)$\Rightarrow$(ii) 
We need to prove the following: if an element
$$
u = a_1 \otimes y_1 + a_2 \otimes y_2 + \sum_{i = 1}^{n-1} c_i \otimes x_i,
$$
where $a_1,a_2,c_1,\dots,c_{n-1} \in \cl A$, is strongly positive in $\cl B \otimes_{\max} NC(n)$ then it
is also strongly positive in $\cl A \otimes_{\max} NC(n)$. Since cones are closed with respect to the order norm,
the desired embedding will then be automatically satisfied. Proposition \ref{prop pos-crit} implies that there exist
positive elements $z_1,\dots,z_{n-1}$ in $\cl B$ such that
$(a_1,a_2,c_1,\dots,c_{n-1},z_1,\dots,z_{n-1})$ is a Riesz decomposition scheme.
Using (i), we conclude that $z_1,\dots,z_{n-1}$ can be chosen from $\cl A$. Finally Proposition \ref{prop pos-crit} implies that
that $u$ is strongly positive in $\cl A \otimes_{\max} NC(n)$.

(ii)$\Rightarrow$(i) Suppose that $a_1,a_2,c_1,\dots,c_{n-1} \in \cl A$ and
$z_1,\dots,z_{n-1}$ are strongly positive in $\cl B$ 
are such that 
the tuple $(a_1,a_2,c_1,\dots,c_{n-1},z_1,\dots,z_{n-1})$ is a Riesz decomposition scheme.
Proposition \ref{prop pos-crit} implies that the element
$u = a_1 \otimes y_1 + a_2 \otimes y_2 - \sum_{i=1}^{n-1} c_i \otimes x_i$
is strongly positive in $\cl B \otimes_{\max} NC(n)$. 
By assumption,
$u$ is strongly positive in $\cl A \otimes_{\max} NC(n)$. 
Now, using Proposition \ref{prop pos-crit}
once again, 
it is easy to see that a Riesz decomposition scheme exists all of whose entries 
belong to $\cl A$.
\end{proof}

\begin{proof}[Proof of Theorem \ref{thm WEPRiesz}] 
By Proposition \ref{prop RD&T}, it suffices to prove
that $\cl A$ has WEP if and only if
$$
\cl A \otimes_{\max} NC(n) \subseteq \cl B(H) \otimes_{\max} NC(n)  
$$
completely order isomorphically for all $n$, and, equivalently for $n=3$. 
This follows from Corollary \ref{p_llp} and 
Theorem \ref{p_eqmma}.
\end{proof}

\begin{remark}
In Theorem \ref{thm WEPRiesz}, $\cl B(H)$ 
can be replaced by any injective C*-algebra containing 
$\cl A$ and, in particular, with the injective envelope $I(\cl A)$ of $\cl A$. 
Thus, for a
unital C*-algebra $\cl A$ the following are equivalent:
\begin{enumerate}
 \item[(i)] $\cl A$ has WEP;
 \item[(ii)] $\cl A$ has the complete $2$-Riesz decomposition property in $I(\cl A)$;
 \item[(iii)] $\cl A$ has the complete $n$-Riesz decomposition property in $I(\cl A)$ 
 for every $n\in \bb{N}$.
\end{enumerate}
The proof is identical to that of Theorem \ref{thm WEPRiesz} 
after noting that the minimal and the maximal tensor products of 
non-commutative cubes with any injective C*-algebra coincide. 
\end{remark}

\begin{corollary}\label{c_ab}
Let $\cl B$ be a unital C*-algebra and $\cl A\subseteq \cl B$
be a unital C*-subalgebra. 
Suppose that $\cl B$ has WEP. Then $\cl A$
has WEP if and only if it has the complete $2$-Riesz decomposition property in $\cl B$.
\end{corollary}

\begin{proof}
By Proposition \ref{prop RD&T}, 
it suffices to prove that $\cl A$ has WEP if and only if 
$$
\cl A \otimes_{\max} NC(3) \subseteq_{\rm coi} \cl B \otimes_{\max} NC(3).
$$
Since $\cl B$ has WEP, Theorem \ref{p_eqmma} implies that
$\cl B \otimes_{\min} NC(3) = \cl B \otimes_{\max} NC(3)$.
Therefore,
the embedding of $\cl A \otimes_{\max} NC(3)$ into $ \cl B \otimes_{\max} NC(3)$ being a complete order inclusion is equivalent to the statement
that the minimal and and the maximal tensor products of $\cl A$ with $NC(3)$ coincides. Thus the result
follows from Theorem \ref{p_eqmma}.
\end{proof}

\begin{corollary}  
(i) \ Every unital C*-algebra $\cl A$ has the $n$-Riesz decomposition property in 
its bidual $\cl A^{**}$, for every $n\in \bb{N}$.

(ii) If $\cl A$ has WEP, then 
$\cl A$ has the complete
$n$-Riesz decomposition in $\cl B$ for every 
C*-algebra $\cl B$ containing $\cl A$ and every
$n\in \bb{N}$.
\end{corollary}

\begin{proof}
(i) This is a direct consequence of \cite[Lemma 6.5]{kavruk--paulsen--todorov--tomforde2010} and Proposition \ref{prop RD&T}. 

(ii) Fix a C*-algebra $\cl B$ with $\cl A \subseteq \cl B$ and $n\geq 2$.
By Theorem \ref{p_eqmma}, 
$\cl A \otimes_{\min} NC(n) = \cl A \otimes_{\max} NC(n)$.
A standard diagramme chase now shows that 
$\cl A \otimes_{\max} NC(n) \subseteq_{\rm coi} \cl B \otimes_{\max} NC(n)$.
Proposition \ref{prop RD&T} shows that 
$\cl A$ has the complete $n-1$-Riesz decomposition property in $\cl B$.
\end{proof}

We will next formulate the Connes Embedding Problem in terms of
the Riesz decomposition property.
Since $C^*(\bb F_2)$ is a residually finite dimensional (for brevity, RFD)
C*-algebra  \cite{choi1980}, there is a C*-algebraic embedding
$$
C^*(\bb F_2) \hookrightarrow \prod_{k=1}^{\infty} M_{n(k)},
$$
for some sequence $(n(k))_{k\in \bb{N}}$ of natural numbers.

\begin{theorem}
Connes' embedding problem has an affirmative solution
if and only if $C^*(\bb F_2)$ has the complete 2-Riesz decomposition
property in the C$^*$-algebra $\prod_{k=1}^{\infty} M_{n(k)}$.
\end{theorem}
\begin{proof} 
The claim follows from Corollary \ref{c_ab} and the fact that 
$\prod_{k=1}^{\infty} M_{n(k)}$ has WEP, or by using Remark 6.6 and the fact that this algebra is injective.
\end{proof}

We will finish this section with a comparison of the 
Riesz decomposition and the Riesz interpolation properties.
Recall that a unital C*-subalgebra $\cl A$ of a C*-algebra $\cl B$ is said to have 
the \textit{relative $(k,m)$-tight
Riesz interpolation property in $\cl B$} (for brevity, the \emph{${\rm TR}(k,m)$-property in $\cl B$})
\cite{kavruk2012} if, whenever
$a_1,\dots,a_m, b_1,\dots,b_k \in\cl A$ are selfadjoint elements,
the existence of a selfadjoint element $x\in \cl B$ satisfying
$$
a_1,\dots,a_m \gg x \gg b_1,\dots,b_k
$$
implies the existence of a selfadjoint element $y\in \cl A$ 
such that $a_1,\dots,a_m \gg y \gg b_1,\dots,b_k$. Likewise,
we say that $\cl A$ has the \emph{complete TR$(k,m)$-property in $\cl B$} if $M_n(\cl A)$ has 
the TR$(k,m)$-property in $M_n(\cl B)$ for every $n\in \bb{N}$.

\begin{theorem}
The following are equivalent, for a unital C*-subalgebra $\cl A\subseteq \cl B(H)$:
\begin{enumerate}
\item[(i)] $\cl A$ has the complete $2$-Riesz decomposition property in $\cl B(H)$;
\item[(ii)] $\cl A$ has the complete ${\rm TR}(2,3)$-property in $\cl B(H)$;
\item[(iii)] $\cl A$ has the complete ${\rm TR}(k,m)$-property in 
$\cl B(H)$ for every $k$ and $m$;
\item[(iv)] $\cl A$ has WEP.
\end{enumerate}
\end{theorem}
\begin{proof}
The assertion is a direct consequence of 
Theorem \ref{thm WEPRiesz} and \cite[Theorem 7.4]{kavruk2012}.
\end{proof}


\begin{thebibliography}{10}


\bibitem{CD} C. D. Aliprantis and R. Tourky,
\newblock Cones and duality, 
\newblock \textit{Volume 84 of Graduate Studies
in Mathematics}, American Mathematical Society, Providence, RI, 2007.



\bibitem{Boca91} 
F. Boca,
\newblock Free products of completely positive maps and spectral sets, 
\newblock \textit{J. Funct. Anal.} 97 (1991), 251--263.



\bibitem{choi1980}
M.-D.~Choi,
\newblock The full C$^*$-algebra of the free group on two generators,
\newblock \textit{Pacific J. Math.} 87 (1980), 41--48.



\bibitem{ChoiEffros} M.-D.~Choi and E.G.~Effros,
\newblock Nuclear C$^*$-algebras and injectivity: the general case,
\newblock\textit{Indiana Univ. Math. J.} 26 (1977), 443--446.

\bibitem{choi--effros1977}
M.-D.~Choi and E.G.~Effros,
\newblock Injectivity and operator spaces,
\newblock \textit{J. Funct. Anal.} 24 (1977), 156--209.



\bibitem{Davidson} K. R. Davidson, 
\newblock C* -algebras by example, 
\newblock \textit{Fields Institute Monographs 6}, American Mathematical Society, Providence, RI, 1996.


\bibitem{EffrosLance} E. G.~Effros and E. C.~Lance,
\newblock Tensor products of operator algebras,
\newblock\textit{Adv. Math.} 25 (1977), 1--34. 



\bibitem{farenick--kavruk--paulsen2011}
D.~Farenick, A.~S. Kavruk, and V.~I. Paulsen,
\newblock C$^*$-algebras with the weak expectation property and a multivariable
  analogue of {A}ndo's theorem on the numerical radius,
\newblock {\em J. Operator Theory}, to appear.


\bibitem{fkpt}
D.~Farenick, A.~S. Kavruk, V.~I. Paulsen, and I.~G. Todorov,
\newblock Operator systems from discrete groups,
\newblock \textit{preprint (arXiv:1209.1152)}, 2012.


\bibitem{farenick--paulsen2011}
D.~Farenick and V.~I. Paulsen.
\newblock Operator system quotients of matrix algebras and their tensor products.
\newblock \textit{Math. Scand.} 111 (2012), 210--243.

\bibitem{fritz0}
T. Fritz, 
\newblock Operator system structures on the unital direct sum of C*-algebras, 
\newblock \textit{preprint (arXiv:1011.1247)}, 2010.

\bibitem{fritz}
T. Fritz,
\newblock Tsirelson's problem and Kirchberg's conjecture,
\newblock \textit{Rev. Math. Phys.} 24 (2012), 1250012 (67 pages).

\bibitem{Han}
K. H. Han, 
\newblock On maximal tensor products and quotient maps of operator systems, 
\newblock \textit{J. Math. Anal. Appl.} 384 (2011), 375--386.

 
\bibitem{kavruk2011}
A.~S. Kavruk,
\newblock Nuclearity related properties in operator systems,
\newblock  \textit{J. Operator Theory}, to appear, \textit{preprint (arXiv:1107.2133)}, 2011.
 
\bibitem{kavruk2012}
A.~S. Kavruk.
\newblock The weak expectation property and Riesz interpolation,
\newblock \textit{preprint (arXiv:1201.5414)}, 2012.


\bibitem{kavruk--paulsen--todorov--tomforde2011}
A.~S. Kavruk, V.~I. Paulsen, I.~G. Todorov, and M.~Tomforde,
\newblock Tensor products of operator systems,
\newblock {\em J. Funct. Anal.}, 261 (2011), 267--299.


\bibitem{kavruk--paulsen--todorov--tomforde2010}
A.~S. Kavruk, V.~I. Paulsen, I.~G. Todorov, and M.~Tomforde,
\newblock Quotients, exactness, and nuclearity in the operator system category,
\newblock\textit{Adv. Math.} 235 (2013), 321--360.


\bibitem{kerrli} D. Kerr and H. Li,
\newblock On Gromov-Hausdorff convergence for operator metric spaces, 
\newblock \textit{J. Operator Theory}, 62 (2009), 83--109.


\bibitem{kirchberg1993} E. Kirchberg, 
\newblock On non-semisplit extensions, tensor products and exactness of group C*-algebras, 
\newblock\textit{Invent. Math.} 112 (1993), 449--489.

\bibitem{ki}
E. Kirchberg, 
\newblock Commutants of unitaries in UHF algebras and functorial properties of exactness,
\newblock\textit{J. Reine Angew. Math.} 452 (1994), 39--77.

\bibitem{L} C. Lance,
\newblock  On nuclear C*-algebras,
\newblock \textit{J. Funct. Anal.} 12 (1973), 157--176.

\bibitem{Lance} C. Lance, 
\newblock Tensor products and nuclear C*-algebras, 
\newblock \textit{Proceedings of Symposia in Pure Mathematics} 
38 (1982), 379-399.


\bibitem{ozawa}
N. Ozawa,
\newblock On the QWEP conjecture,
\newblock \textit{Int. J. Math.} 15 (2004),  501--530.


\bibitem{ozawa2013}
N. Ozawa,
\newblock About the Connes embedding conjecture,
\newblock \textit{Jpn. J. Math.} 8 (2013), 147--183.

\bibitem{PaulsenBook} V. I. Paulsen, Completely bounded maps and operator algebras,
\newblock\textit{Cambridge Studies in Advanced Mathematics 78}, Cambridge University Press, 2002.


\bibitem{paulsen--todorov--tomforde2011}
V.~I. Paulsen, I.~G. Todorov, and M.~Tomforde,
\newblock Operator system structures on ordered spaces,
\newblock \textit{Proc. London Math. Soc.} 102 (2011), 25--49.


\bibitem{pisier_k}
G. Pisier, 
\newblock A simple proof of a theorem of Kirchberg's and related results on C$^*$-norms,
\newblock\textit{J. Operator Theory} 35 (1996), 317--335.


\bibitem{pisier_intr}
G. Pisier, 
\newblock Introduction to operator space theory,
\newblock\textit{Cambridge University Press}, 2003.



\bibitem{Riesz} F. Riesz,
\newblock Sur quelques notions fondamentales dans la th\'eorie g\'en\'erale des op\'erations lin\'eaires,
\newblock \textit{Ann. of Math.} 41 (1940), 174--206.



\end{thebibliography}
\end{document}